\documentclass{amsart}
\newcommand{\abs}[1]{\lvert#1\rvert}

\usepackage{amssymb}
\usepackage{amsmath}
\usepackage{amsfonts}
\usepackage{tikz-cd}
\usepackage{amsthm}
\usepackage{enumerate}
\usepackage{mathrsfs}
\usetikzlibrary{decorations.pathmorphing}
\usepackage[utf8]{inputenc}
\usepackage[T1]{fontenc}
\usepackage{indentfirst}
\usepackage{appendix}

\newtheorem{theorem}{Theorem}[section]
\newtheorem{lemma}[theorem]{Lemma}

\theoremstyle{definition}
\newtheorem{define}[theorem]{Definition}
\newtheorem{example}[theorem]{Example}

\newtheorem{convention}[theorem]{Convention}
\newtheorem{corollary}[theorem]{Corollary}
\newtheorem{proposition}[theorem]{Proposition}

\theoremstyle{remark}
\newtheorem{remark}[theorem]{Remark}

\numberwithin{equation}{section}

\DeclareMathOperator{\Hom}{Hom}
\DeclareMathOperator{\Ext}{Ext}

\DeclareMathOperator{\Der}{Der}

\DeclareMathOperator{\ann}{ann}

\DeclareMathOperator{\gr}{gr}

\DeclareMathOperator{\Variety}{V}
\DeclareMathOperator{\xpoint}{\mathfrak{x}}

\DeclareMathOperator{\Ch}{Ch}
\DeclareMathOperator{\rad}{rad}
\DeclareMathOperator{\rel}{rel}
\DeclareMathOperator{\Weyl}{A_{\emph{n}}(\mathbb{C})}
\DeclareMathOperator{\Support}{Supp}
\DeclareMathOperator{\Spec}{Spec}
\DeclareMathOperator{\mSpec}{mSpec}
\DeclareMathOperator{\codim}{codim}
\DeclareMathOperator{\red}{red}

\author{Daniel Bath}
\title[A Note on Bernstein--Sato Varieties]{A Note on Bernstein--Sato Varieties for Tame Divisors and Arrangements}
\address{Department of Mathematics, Purdue University, Lafayette, Indiana, USA.}
\email{dbath@purdue.edu}
\thanks{This research was supported by a Purdue Research Foundation grant from Purdue University and by a Simons Foundation Collaboration Grant for Mathematicians \#580839.}

\begin{document}

\maketitle

\begin{abstract}
    For strongly Euler-homogeneous, Saito-holonomic, and tame analytic germs we consider general types of multivariate Bernstein--Sato ideals associated to arbitrary factorizations of our germ. We show that these ideals are principal and the zero loci associated to different factorizations are related by a diagonal property. If, additionally, the divisor is a hyperplane arrangement, we obtain nice estimates for the zero locus of its Bernstein--Sato ideal for arbitrary factorizations and show the Bernstein--Sato ideal attached to a factorization into linear forms is reduced. As an application, we independently verify and improve upon an estimate of Maisonobe's regarding standard Bernstein--Sato ideals for reduced, generic arrangements: we compute the Bernstein--Sato ideal for a factorization into linear forms and we compute its zero locus for other factorizations.
\end{abstract}

\iffalse
\tableofcontents
\fi

\section{Introduction}

Let $X$ be a smooth analytic space or a $\mathbb{C}$-scheme with dimension $n$ and consider the collection $F = (f_{1}, \dots, f_{r})$ of regular, analytic functions $f_{k} \in \mathscr{O}_{X}$. Throughout, $\mathfrak{x}$ will denote a point in $X$, $f = f_{1} \cdots f_{r}$, and $\mathscr{D}_{X}$ will signify the sheaf of $\mathbb{C}$-linear differential operators. For 
\[
\textbf{a} = (a_{1}, \dots, a_{r}) \in \mathbb{N}^{r}
\]
there is a $\mathscr{D}_{X}[S] := \mathscr{D}_{X}[s_{1}, \dots, s_{r}]$-module (this is a standard polynomial ring extension where the $s_{k}$ are central variables)
\[
\mathscr{D}_{X}[S] F^{S + \textbf{a}} := \mathscr{D}_{X}[s_{1}, \dots, s_{r}]f_{1}^{s_{1} + a_{1}} \cdots f_{r}^{s_{r} + a_{r}},
\]
where the action of a differential operator on 
\[
F^{S+\textbf{a}} := f_{1}^{s_{1} + a_{1}} \cdots f_{r}^{s_{r} + a_{r}}
\]
is given by a formal application of the chain and product rule. As in the classical setting, where $F$ is just one function, we have a \emph{functional equation}
\[
b(S) F^{S} = P F^{S+\textbf{a}}
\]
with $b(S) \in \mathbb{C}[S] := \mathbb{C}[s_{1}, \dots, s_{r}]$ and $P \in \mathscr{D}_{X}[S]$ as well as a local version where each $f_{k}$ is regarded as its germ at $\xpoint.$ Generalizing the Bernstein--Sato polynomial, we have the multivariate \emph{Bernstein--Sato ideal} $B_{F, \mathfrak{x}}^{\textbf{a}}$ which consists of all the polynomials $b(S) \in \mathbb{C}[S]$ satisfying said local functional equation. Explicitly,
\[
B_{F, \mathfrak{x}}^{\textbf{a}} := \{ b(S) \in \mathbb{C}[S] \mid b(S) F^{S} \in \mathscr{D}_{X,\mathfrak{x}}[S] F^{S+ \textbf{a}} \}.
\]
Finally, let $Z(B_{F,\mathfrak{x}}^{\textbf{a}})$ be the zero locus cut out by $B_{F,\mathfrak{x}}^{\textbf{a}}$:
\[
Z(B_{F,\mathfrak{x}}^{\textbf{a}}) := \Variety(\text{rad}(B_{F,\mathfrak{x}}^{\textbf{a}})).
\]

Recently a lot of work has been done on multivariate Bernstein--Sato ideals and their zero loci. Let $\textbf{1} = (1, \dots, 1)$. While it had been shown by Sabbah and Gyoja, cf. \cite{SabbahI, Gyoja}, that the codimension one components of $Z(B_{F,\mathfrak{x}}^{\textbf{1}})$ are rational hyperplanes, Maisonobe proved in \cite{Maisonobe} that any component of larger codimension can be translated into a codimension one component by an element of $\mathbb{Z}^{r}$. He did this by proving many useful homological properties of 
\[
\frac{\mathscr{D}_{X,\mathfrak{x}}[S] F^{S}}{\mathscr{D}_{X,\mathfrak{x}}[S]F^{S+\textbf{1}}}
\]
and analyzing various attached associated graded objects. Using some of these results, in \cite{ZeroLociI} Budur, Van der Veer, Wu, and Zhou completed a proof of a conjecture of Budur's from \cite{BudurLocalSystems}. Namely, they showed that if $F = (f_{1}, \dots, f_{r})$ corresponds to the factorization $f = f_{1} \cdots f_{r}$, then exponentiating $Z(B_{F,\mathfrak{x}}^{\textbf{1}})$ computes the rank one local systems on $X \setminus \Variety(f)$ with nontrivial cohomology. (See also section 3.29 of \cite{BudurLocalSystems}.) This can be viewed as a generalization of Kashiwara and Malgrange's result in \cite{Malgrange}, \cite{Kashiwara} that exponentiating the roots of the Bernstein--Sato polynomial computes the eigenvalues of the algebraic monodromy on nearby Milnor fibers. Indeed, in the multivariate setting, the role of the nearby cycle functor is replaced by the Sabbah specialization complex, cf. \cite{BudurLocalSystems, BudurCohomologySupportLoci, ZeroLociI}. Most of these results have been generalized to $Z(B_{F,\mathfrak{x}}^{\textbf{a}})$ for $\textbf{a} \in \mathbb{N}^{r}$ arbitrary in \cite{ZeroLociII}.

Our main objective is to identify geometric conditions on $\text{Div}(f)$ so that $Z(B_{F,\mathfrak{x}}^{\textbf{a}})$ and/or $B_{F, \mathfrak{x}}^{\textbf{a}}$ has a particularly nice structure. The hypotheses concern the \emph{logarithmic derivations} $\Der_{X}(-\log f)$ of $\text{Div}(f)$, that is, the derivations that preserve $f$, \iffalse(or any other defining equation of $\text{Div}(f)$) \fi as well as a sort of dual object introduced by Saito in \cite{SaitoLogarithmicForms}: the \emph{logarithmic differential forms} $\Omega_{X}^{\bullet}(\log f)$. We will often assume the following for $\text{Div}(f)$: it is strongly Euler-homogeneous (locally everywhere it has a singular derivation that acts as the identity on $f$); it is Saito-holonomic (the stratification of $X$ by the logarithmic derivations is locally finite); it is tame ($\Omega_{X}^{\bullet}(\log f)$ satisfies a scaling projective dimension bound). This set of conditions was first considered by Walther in \cite{uli} in the univariate setting ($F = (f)$) and we considered them in the multivariate setting ($F = f_{1} \cdots f_{r}; r > 1$) in \cite{Bath1}.

While in general $B_{F,\mathfrak{x}}^{\textbf{a}}$ may not be principal (cf. \cite{BahloulOakuLocalBSIdeals, BrianconMaynadierPrincipality}), using some of our results from \cite{Bath1} as well as a crucial idea of Maisonobe's from \cite{MaisonobeFreeHyperplanes}, we can obtain the following in Corollary \ref{cor- tame implies BS principal}:

\begin{theorem} \label{thm-intro purely codim one}
Suppose that $f$ is strongly Euler-homogenous, Saito-holonomic, and tame, $F$ corresponds to any factorization $f = f_{1} \cdots f_{r}$, and $f_{1}^{a_{1}} \cdots f_{r}^{a_{r}}$ is not a unit. Then $B_{F, \mathfrak{x}}^{\textbf{a}}$ is principal.
\end{theorem}

This follows from the much more general statement Theorem \ref{thm- purity implies principality} about $(n+1)$-pure (Definition \ref{def-purity}), relative holonomic (Definition \ref{def- relative holonomic}) $\mathscr{D}_{X,\mathfrak{x}}[S]$-modules, that we spotlight since it is of general interest:

\begin{theorem} \label{thm- intro purity implies principality} Suppose that $M$ is a finite $\mathscr{D}_{X,\mathfrak{x}}[S]$-module that is relative holonomic and $(n+1)$-pure. Then its Bernstein--Sato ideal $B_{M}$, i.e. its $\mathbb{C}[S]$-annihilator, is principal.
\end{theorem}

\iffalse

if $M$ is a $(n+1)$-pure (Definition \ref{def-purity}), relative holonomic (Definiton \ref{def- relative holonomic}) $\mathscr{D}_{X,\mathfrak{x}}[S]$-module, then its $\mathbb{C}[S]$-annihilator is principal. 

\fi

The hurdle becomes verifying that, under the hypotheses of Theorem \ref{thm-intro purely codim one}, 
\[
\mathscr{D}_{X,\mathfrak{x}}[S] F^{S} / \mathscr{D}_{X,\mathfrak{x}}[S]F^{S+\textbf{a}}
\]
satisfies the subtle homological property of purity. In Theorem \ref{thm- F^S / F^S+1 CM} we prove something much stronger: it has only one nonvanishing dual Ext-module, i.e. it is $(n+1)$-Cohen--Macaulay (Definition \ref{def- CM for Auslander regular}).

In Lemma 4.20 of \cite{BudurLocalSystems}, Budur shows that information about the multivariate Bernstein--Sato ideal of a finer factorization of $f$ gives (potentially incomplete) data about the multivariate Bernstein--Sato ideal of a coarser factorization of $f$. Let $B_{f, \mathfrak{x}} \subseteq \mathbb{C}[s]$ be the ideal generated by the usual Bernstein--Sato polynomial of $f$ at $\mathfrak{x}$. Under the hypotheses of Theorem \ref{thm-intro purely codim one} we obtain in Theorem \ref{thm-intersecting with diagonal} a strengthening of Lemma 4.20 in loc. cit. While this result works for general $\textbf{a}$ and various factorizations, here is a special case:

\iffalse For the sake of the Introduction's readability, we will restrict to $\textbf{a} = \textbf{1}$ and focus on comparing the factorizations $F = (f_{1}, \dots, f_{r})$ of $f$ and the trivial factorization $f = f.$ Let $B_{f, \mathfrak{x}} \subseteq \mathbb{C}[s]$ be the ideal generated by the Bernstein--Sato ideal of $f$. Then the following is a special case of Theorem \ref{thm-intersecting with diagonal}: \fi

\begin{theorem} \label{thm-intro intersect diagonal}

Suppose that $f$ is strongly Euler-homogenous, Saito-holonomic, and tame and $F$ corresponds to any factorization $f = f_{1} \cdots f_{r}$. Then
\[
Z(B_{f, \mathfrak{x}})= Z(B_{F, \mathfrak{x}}^{\textbf{1}}) \cap \{ s_{1} = \cdots = s_{r} \}.
\]
Here points on the diagonal of $\mathbb{C}^{r}$ are naturally identified with points in $\mathbb{C}$.

\end{theorem}

This has many applications for arrangements. In Corollary \ref{cor-free arrangements} we finish the computation of $Z(B_{F,0}^{\textbf{1}})$ and $Z(B_{f,0})$ for $f$ a central, possibly non-reduced, free arrangement and $F$ an arbitrary factorization by covering the subset of non-reduced cases not included in Theorem 1.4 of \cite{Bath2}. This finishes our certification, started in loc. cit., that the roots of the Bernstein--Sato polynomial of a free hyperplane are always combinatorial, regardless of any reducedness assumption. In Example \ref{ex-not combinatorial} we extend Walther's result that the b-function of a hyperplane arrangement is not necessarily combinatorial to the multivariate case. In Corollary \ref{cor- multivariate root bounds} we extend Saito's bounds on $Z(B_{f,0})$ from Theorem 1 of \cite{SaitoArrangements} to the tame, multivariate setting. The later bounds are somewhat precise since the hyperplanes in $Z(B_{F,0})$ corresponding to roots of $Z(B_{f,0}) \cap [-1,0)$ are combinatorially determined by Theorem 4.11, Theorem 1.3 of \cite{Bath2}.

\iffalse
In Theorem 1.4 of \cite{Bath2} we computed $Z(B_{f,0})$ for $f$ a power of a central, reduced, free hyperplane arrangement and $Z(B_{F,0}^{\textbf{1}})$ many factorizations $F$ of a central, (possibly non-reduced) free hyperplane arrangement. We use this diagonal property to obtain the expected formula for $Z(B_{f, 0})$ for any central, free hyperplane arrangement in Corollary \ref{cor-free arrangements}. Furthermore, we easily extend Walther's result in \cite{uli} (Example 5.10 therein) that the b-function of a hyperplane arrangement need not be combinatorially determined extends to the multivariate case: for $F$ an arbitrary factorization of a hyperplane arrangement, $Z(B_{F,\mathfrak{x}}^{\textbf{1}})$ need not be combinatorially determined, cf. Remark \ref{ex-not combinatorial}.
\fi

As this paper was being completed, Wu proved, using different methods, the same diagonal result from Theorem \ref{thm-intro intersect diagonal} under more general hypotheses in Theorem 1.1 of \cite{WuHyperplanes}. While ours hypotheses are stronger, they are more geometric and verifiable; moreover, that ours assumptions imply his is not at all obvious, cf. Theorem \ref{thm- F^S / F^S+1 CM}. Using this diagonalization, he considered central, reduced, free arrangements and recovered in Theorem 1.2 of \cite{WuHyperplanes} a special case of our formula for $Z(B_{f,0})$ that originally appeared in Theorem 1.4 of \cite{Bath2}. \iffalse Moreover, while he outlines how to handle the extant cases from loc. cit.: $f$ not a power of a reduced arrangement.\fi
%removed this part since you need my duality fml for non-reduced things

%In general, \iffalse even in the case where $F$ is a factorization into irreducibles, \fi $B_{F,\mathfrak{x}}^{\textbf{1}}$ may not be principal, see \cite{BrianconMaynadierPrincipality, BahloulOakuLocalBSIdeals}. \iffalse where criterion and counter-examples are given. \fi However, Theorem \ref{thm-tame principal BS ideal} shows this happens when $f$ is tame hyperplane arrangement factored into linear forms: \iffalse , then we can strengthen Theorem \ref{thm-intro purely codim one} in is a hyperplane arrangement, then we can we generalize an argument of Maisonobe's in  of \cite{MaisonobeFreeHyperplanes} to find a reduced hyperplane arrangement lying in $B_{F,\mathfrak{x}}^{\textbf{a}}$. Using  we can then prove: \fi

Additionally, essentially because hyperplane arrangements have a factorization $F$ into linear forms, the Bernstein--Sato ideal attached to this special factorization has a particularly good property:

\begin{theorem} \label{thm-intro principal ideal}
Suppose that $f$ is a tame, possibly non-reduced hyperplane arrangement, $F$ corresponds to a factorization of $f$ into linear forms, and $f_{1}^{a_{1}} \cdots f_{r}^{a_{r}}$ is not a unit. Then $B_{F}^{\textbf{a}} = \rad (B_{F}^{\textbf{a}})$.
\end{theorem}

\iffalse For tame arrangements we also obtain in Corollary \ref{cor- multivariate root bounds} a multivariate version of Saito's Theorem 1 in \cite{SaitoArrangements} that the roots of the b-function live in $(-2 + \frac{1}{d}, 0)$. \fi

Finally we consider central, reduced generic arrangements $f$ of $d$ hyperplanes in $\mathbb{C}^{n}$. In \cite{WaltherGeneric}, Walther obtained a formula for the Bernstein--Sato polynomial for $f$. For $F$ a factorization into linear forms, Maisonobe found in \cite{MaisonobeGeneric} an element of the global Bernstein--Sato ideal $B_{F}^{\textbf{1}}$ for $d > n+1$ and computed this ideal for $d = n+1$. However, completely computing this ideal for $d > n+1$ remained open. Using Walther's formula, in Theorem \ref{thm- generic BS ideal formulae}, we consider any $d > n$, independently verify Maisonobe's analysis of the $d=n+1$ case, and explicitly determine this ideal for any $d > n$: \iffalse $B_{F}^{\textbf{1}}$ for factorizations into linear forms and compute its zero locus otherwise: \fi

\begin{theorem} \label{thm-intro my generic formulae}
Let $f$ be the reduced defining equation of a central, generic hyperplane arrangement in $\mathbb{C}^{n}$ with $d = \deg f > n$. If $F$ corresponds to a factorization of $f$ into irreducibles, then 
\begin{equation} \label{eqn-intro my generic BS generator}
B_{F}^{\textbf{1}} = \mathbb{C}[S] \cdot \prod_{k=1}^{d} (s_{k} + 1) \prod_{i=0}^{2d - n -2}(\sum_{k=1}^{d} s_{k} + i + n).
\end{equation}
If $F=(f_{1}, \dots, f_{r})$ corresponds to some other factorization, then $B_{F}^{\textbf{1}}$ is principal and
\begin{equation} \label{eqn-intro my generic BS zero locus}
Z(B_{F}^{\textbf{1}}) = \left(\bigcup_{k=1}^{r} \{ s_{k} + 1 = 0 \}\right) \bigcup \left( \bigcup_{i = 0}^{2d - n - 2} \{ \sum_{k=1}^{r} d_{k}s_{k} + i + n = 0 \} \right).
\end{equation}
\end{theorem}

In Appendix A we document some basic properties of the logarithmic differential forms for non-reduced divisors. This is our attempt to fill an apparent gap in the literature, as previous texts seem to focus only on the reduced case. 

We would like to thank Luis Narv{\'a}ez Macarro, Uli Walther, Harrison Wong, and Lei Wu for helpful comments and conversations, as well as the referee for his/her suggestions that helped clarify and improve the paper.

\section{Preliminaries}

Here we catalogue the various facts needed to prove the paper's first set of results. The three subsections are divided up as follows: first, we collect basic lemmas about Auslander regular rings from \cite{Bjork, ZeroLociI}; second, we review two different filtrations on $\mathscr{D}_{X}[S]$ as well as a fundamental structure theorem from \cite{ZeroLociI, ZeroLociII}; third, we introduce the geometric hypotheses on \text{Div}(f) we assume later in the paper. 

\iffalse

We study the multivariate Bernstein--Sato ideal at $\xpoint \in X$ for $f$ an analytic function on a pure manifold of dimension $n$. We assume $f$ is strongly Euler-homogeneous, Saito-holonomic, and tame. Using standard facts from \cite{Bjork} we prove in Theorem \ref{thm- F^S / F^S+1 CM} that 
\[
\Ext_{\mathscr{D}_{X,\xpoint}[S]}^{k}(\mathscr{D}_{X,\xpoint}[S]F^{S} / \mathscr{D}_{X,\xpoint}[S] F^{S+1}, \mathscr{D}_{X}[S]) = 0 \text{ for all } k \neq n+1.
\]
This implies by work of Maisonobe that the reduced locus cut out by the multivariate Bernstein--Sato ideal is purely codimension one, cf. Corollary \ref{cor - BS zero locus codim one}. In Theorem \ref{thm-intersecting with diagonal}, we use a crucial result of \cite{ZeroLociI}, to prove that intersecting said reduced locus with the diagonal gives the roots of the b-polynomial of $f$. Finally we answer (in the affirmative) a question posed in \cite{Bath1} in Corollary \ref{cor-three equivalent things nabla}.

\fi

\subsection{Lemmas about Auslander Regular Rings} \text{ }

We consider (noncommutative) rings $A$ that are positively filtered such that $\gr A$ (the associated graded ring induced by the positive filtration) is a commutative, Auslander regular, Noetherian ring. By A.III.1.26 and A.IV.4.15 of \cite{Bjork} this implies $A$ is Zariskian filtered and Auslander regular. Let $M$ be a nonzero finitely generated left $A$-module. In this subsection we are mostly interested in the $A$-module $\Ext_{A}^{k}(M,A)$ as well as the interplay between $M$ and $\gr M$ (for a choice of good filtration on $M$). Note that because $A$ is Zariskian filtered, if $M$ is nonzero then $\gr M$ is also nonzero (for any good filtration on $M$).

Before embarking let us emphasize the utility: both the \emph{total order filtration} or \emph{relative order filtration} of the next subsection are positive filtrations on $\mathscr{D}_{X,\mathfrak{x}}[S]$ whose associated graded objects are commutative, Auslander regular, and Noetherian. So the facts below are germane.

\begin{define} \label{def-purity}
For $A$ an Auslander regular ring and $0\neq M$ a finitely generated $A$-module, the \emph{grade} of $M$ is the smallest integer $k$ such that $\Ext_{A}^{k}(M,A) \neq 0.$ We say $M$ is \emph{pure} of grade $j(M)$ if every nonzero submodule of $M$ has grade $j(M)$. 
\end{define}

The grade can be computed on the associated graded side:

\begin{proposition} \normalfont{(A.IV.4.15 \cite{Bjork})} \label{prop- M pure grade is associated graded grade}
Let $A$ be a positively filtered ring such that $\gr A$ is a commutative, Noetherian, and Auslander regular ring and further assume that $0 \neq M$ a finitely generated left $A$-module. Then, for any good filtration on $M$, $j(M) = j(\gr M).$
\end{proposition}

The following lemma gives criteria to check purity:

\iffalse
\begin{define}
For $A$ an Auslander regular ring and $0\neq M$ a finitely generated $A$-module, we say $M$ is \emph{pure} of grade $j(M)$ if every nonzero submodule of $M$ has grade $j(M)$. 
\end{define}
\fi

\begin{lemma} \label{lemma-criterion for purity} \normalfont{ (A.IV.2.6, A.IV.4.11 \cite{Bjork})}
For $A$ an Auslander regular ring and $0 \neq M$ a finitely generated left $A$-module, then
\begin{enumerate}[(i)]
\item $\text{\normalfont Ext}_{A}^{j(M)}(M, A)$ is a pure $A$-module of grade $j(M)$;
\item $M$ is a pure $A$-module of grade $j(M)$ if and only if $\text{\normalfont Ext}_{A}^{k}(\text{\normalfont Ext}_{A}^{k}(M,A) = 0$ for all $k \neq j(M)$.
\item Suppose further that $A$ is positively filtered such that $\gr A$ is commutative, Noetherian, and Auslander regular. If $M$ is a pure $A$-module, then there exists a good filtration on $M$ such that $\gr M$ is a pure $\gr A$-module of grade $j(M)$.
\end{enumerate}
\end{lemma}

In \cite{ZeroLociI} the authors introduce a noncommutative generalization of Cohen--Macaulay modules which will prove very useful for us throughout this paper.

\begin{define} \label{def- CM for Auslander regular}
For $A$ an Auslander regular ring and $0 \neq M$ a finitely generated left $A$-module, we say $M$ is \emph{$j$-Cohen--Macualay} if 
\[
\text{Ext}_{A}^{k}(M, A) = 0 \text{ precisely when } k \neq j.
\]
\end{define}

Purity and being $j$-Cohen--Macaulay are related by the following:

\begin{proposition} \label{prop-CM implies pure}
For $A$ an Auslander regular ring and $0 \neq M$ a finitely generated left $A$-module, if $M$ is $j$-Cohen--Macaulay then $M$ is pure of grade $j$.
\end{proposition}
\begin{proof}
This is immediate by (i) and (ii) of Lemma \ref{lemma-criterion for purity}.
\end{proof}

A crucial technique of this paper will be to transfer Cohen--Macaulay properties of $\gr M$ to Cohen--Macaulay properties of $M$. The next lemma will be helpful. 

\begin{lemma} \label{lemma-good filtration on ext} \normalfont{(A.IV.4.5 \cite{Bjork})}
Let $A$ be a positively filtered ring and $0 \neq M$ a finitely generated left $A$-module. Then for a good filtration on $M$, there is a good filtration on $\Ext_{A}^{k}(M,A)$ so that $\gr \Ext_{A}^{k}(M,A)$ a subquotient of $\Ext_{\gr A}^{k}(\gr M, \gr A).$
\end{lemma}

\iffalse

\begin{proposition} \label{prop-cyclic gr ann CM}
Suppose $A$ is a positively filtered ring such that $\gr A$ is a commutative, Noetherian, Auslander regular ring. Suppose $M$ is a cyclic, left $A$-module with $M = A \cdot m$. If there is a good filtration on $M$ such that $\gr A / \gr (\ann_{A} m)$ is a $j$-Cohen--Macaulay $\gr A$-module, then $M$ is a $j$-Cohen--Macaulay $A$-module.
\end{proposition}

\begin{proof}
By hypothesis, $\gr M \simeq \gr A / \gr (\ann_{A} m)$ (with respect to the promised good filtration). \iffalse So $\gr M$ is a Cohen--Macaulay $\gr A$-module of co-dimension $j$. It is routine to check the Cohen-Macaulay condition implies that $\Ext_{\gr A}^{k}(\gr M, \gr A)$ vanishes for all $k \neq j$. \fi By Lemma \ref{lemma-good filtration on ext}, there is a good filtration (for each $k$) on $\Ext_{A}^{k}(M, A)$ such that $\gr \Ext_{A}^{k}(M,A)$ is a subquotient of $\Ext_{\gr A}^{k}(\gr M, \gr A)$. So $\gr \Ext_{A}^{k}(M,A)$ and thus $ \Ext_{A}^{k}(M,A)$ vanish for all $k \neq j$.
\end{proof}

\fi

In certain commutative settings, grade of $M$ and the purity of $M$ have geometric interpretations:

\begin{proposition} \label{prop- pure iff associated prime same heights} (Proposition 4.5.1 \cite{ZeroLociI})
Suppose $A$ is a commutative, regular, Noetherian domain and $0 \neq M$ a finitely generated $A$-module. Then 
\begin{enumerate}
    \item If $\dim A_{\mathfrak{m}}$ is the same for all $\mathfrak{m} \in \mSpec A$, then $j(M) + \dim M = \dim A$.
    \item $M$ is a pure $A$-module if and only if every associated prime of $M$ is minimal and $\dim A_{\mathfrak{p}} = j(M)$ for all minimal primes $\mathfrak{p}$ of $M$.
\end{enumerate}
\end{proposition}

\iffalse
\begin{proposition}
Let $A$ be a positively filtered Noetherian ring and $M$ a finitely generated left $A$-module. Then the following hold:
\begin{enumerate}[(i)]
\item hi
\end{enumerate}
\end{proposition}
\fi

Finally, in our setting there is a notion of characteristic variety (defined entirely similarly to the $\mathscr{D}_{X}$-module version) that is integral to this paper. 
 
\begin{define} \label{def-characteristic variety}
Let $A$ be a positively filtered, Noetherian ring such that $\gr A$ is a commutative regular Noetherian ring. Let $0 \neq M$ be a finitely generated left $A$-module and $\Gamma$ a good filtration on $M$. Define the \emph{characteristic variety} $\Ch (M)$ of $M$ by
\[ 
\Ch (M) := \Variety (\rad ( \ann_{\gr A} \gr M))
\]
where $\gr M$ is defined with respect to $\Gamma$. In fact, $\Ch (M)$ does not depend on the choice of good filtration of $M$.
\end{define}

\subsection{Introduction to $\mathscr{D}_{X,\xpoint}[S] F^{S} / \mathscr{D}_{X,\xpoint}[S]F^{S+\textbf{a}}$} \text{ }

\iffalse Let $X$ be a smooth analytic manifold of dimension $n$ with analytic structure sheaf $\mathscr{O}_{X}$ and consider a regular $f \in \mathscr{O}_{X}$. As always, let $\mathscr{D}_{X}$ be the sheaf of $\mathbb{C}$-linear differential operators with $\mathscr{O}_{X}$-coefficients and let $\mathscr{D}_{X}[S]$ be the polynomial ring extension $\mathscr{D}_{X}[s_{1}, \dots, s_{r}]$ where the $s_{k}$ are central variables. We will be utilize two different filtrations on $\mathscr{D}_{X}[S]$.
\fi

Here we introduce two filtrations on $\mathscr{D}_{X}[S]$ that both extend the standard order filtration on $\mathscr{D}_{X, \mathfrak{x}}[S]$. We also recall some of the fundamental results about $\mathscr{D}_{X}[S]F^{S} / \mathscr{D}_{X}[S] F^{S + \textbf{a}}$ proved in \cite{ZeroLociII} (see also \cite{Maisonobe}).

\begin{define} \label{def-total order filtration}
The \emph{total order filtration} on $\mathscr{D}_{X}[S]$ is defined by extending the order filtration on $\mathscr{D}_{X}$ to $\mathscr{D}_{X}[S]$ by giving each $s_{k}$ weight one. More precisely, in local coordinates $(x, \partial)$ near $\xpoint$ we can write any element of $\mathscr{D}_{X, \xpoint}[S]$ as $\sum_{\textbf{v}, \textbf{w}} g_{\textbf{v}, \textbf{w}} \partial^{\textbf{v}} s^{\textbf{w}}$ where $g_{\textbf{v}, \textbf{w}} \in \mathscr{O}_{X,\xpoint}$, $\textbf{v} \in \mathbb{Z}_{\geq 0}^{n}$, $\textbf{w} \in \mathbb{Z}_{\geq 0}^{r}$, $\partial^{\textbf{v}} = \partial_{1}^{u_{1}} \cdots \partial_{n}^{u_{n}}$, and $s^{\textbf{w}} = s_{1}^{w_{1}} \cdots s_{r}^{w_{r}}$. The elements of order at most $k$ are those that admit an expression above satisfying $\abs{\textbf{v}} + \abs{\textbf{w}} = v_{1} + \cdots + v_{n} + w_{1} + \cdots + w_{r} \leq k$ for each summand. \iffalse the elements of order at most $k$ by $F_{(0,1,1)}^{k}$ and \fi The associated graded ring is $\gr^{\sharp}(\mathscr{D}_{X}[S]).$
\end{define}

\begin{define}
The \emph{relative order filtration} on $\mathscr{D}_{X}[S]$ is defined by extending the order filtration on $\mathscr{D}_{X}$ to $\mathscr{D}_{X}[S]$ by giving each $s_{k}$ weight zero. Using the expression defined above, elements of order at most $k$ satisfy $\abs{\textbf{v}} \leq k$. The \iffalse elements of order at most $k$ under the relative order filtration are denoted by $F_{(0,1,0)}^{k}$ and the \fi associated graded ring is $\gr^{\rel}(\mathscr{D}_{X}[S])$. 
\end{define}

Under either the total order filtration or the relative order filtration, the associated graded rings satisfy the following isomorphisms locally:
\[
\gr^{\sharp}(\mathscr{D}_{X, \mathfrak{x}}[S]) \simeq \mathscr{O}_{X,\mathfrak{x}}[y_{1}, \dots, y_{n}][s_{1}, \dots, s_{r}] \simeq \gr^{\rel}( \mathscr{D}_{X, \mathfrak{x}}[S]).
\]
Here, $y_{i}$ corresponds to the principal symbol of $\partial_{i}$ (after picking coordinates) under the appropriate grading. Since either the relative or total order filtration is a positive filtration, this identification means we may apply the techniques from the previous subsection to $\mathscr{D}_{X, \xpoint}[S]$. (See \cite{Bjork} A.IV.3.6.) Now let us solidify some notation related to $\mathscr{D}_{X}[S]F^{S}$: \iffalse In particular, $\mathscr{D}_{X}[S]$ is Auslander regular, either filtration is a positive filtration, \fi

\iffalse
\begin{remark}
In \cite{Bath1, Bath2} we used slightly different notation to denote these filtrations. In loc. cit. the associated graded object attached to total order filtration is given by $\gr_{(0,1,1)}(-)$ instead of $\gr^{\sharp}(-)$; the associated graded object attached to the relative order filtration is given by $\gr_{(0,1,0)}(-)$ instead of $\gr^{\rel}(-)$.
\end{remark}
\fi

\begin{convention} \label{convention} Given a factorization $F = (f_{1}, \dots, f_{r})$ of $f$ and $\textbf{a} \in \mathbb{N}^{r}$, we denote by $f^{\textbf{a}}$ the product $f_{1}^{a_{1}} \cdots f_{r}^{a_{r}}$. The multivariate Bernstein--Sato ideal of $F$ associated to $\textbf{a}$ at $\xpoint$ is denoted by $B_{F,\xpoint}^{\textbf{a}}$; the ideal generated by the \emph{Bernstein--Sato polynomial} (or b-function) of $f$ associated to $a \in \mathbb{N}$ at $\xpoint$ is $B_{f,\xpoint}^{a}$. When $a = 1$ we just write $B_{f,\xpoint}$. When $F$ (or $f$) is global algebraic, we reserve $B_{F}^{\textbf{a}}$ and $B_{f}^{a}$ for the global algebraic versions of these objects, see subsection 3.3. By ``multivariate'' we mean $r > 1$; by ``univariate'' we mean $r = 1.$ The reader should note that in \cite{Bath1, Bath2} we used $\gr_{(0,1,1)}(-)$ and $\gr_{(0,1,0)}(-)$ instead of $\gr^{\sharp}(-)$ and $\gr^{\rel}(-)$ for the associated graded objects attached to the total order and relative order filtrations. In those papers we use $\Variety(-)$ with respect to Bernstein--Sato ideals in the same way as we use $Z(-)$ here: to denote the zero locus. In this paper we use $\Variety(-)$ for elements of $\mathscr{O}_{X}$ or when the input is reduced.
\end{convention}

\begin{define}
Consider a finitely generated $\mathscr{D}_{X}[S]$-module $M$. We denote its characteristic variety, cf. Definition \ref{def-characteristic variety}, with respect to the relative order filtration by $\Ch^{\text{rel}}(M)$; with respect to the total order filtration by $\Ch^{\sharp}(M)$.
\end{define}

In \cite{Maisonobe}, Maisonobe introduce a criterion on a $\mathscr{D}_{X}[S]$-module he calls \emph{major\'e par une lagrangienne}. This guarantees the relative characteristic variety has a very nice product structure. We adopt the verbiage of \cite{ZeroLociI} where this is generalized slightly:

\begin{define} \label{def- relative holonomic} \normalfont{(Definition 3.2.3, \cite{ZeroLociI})}
A finitely generated $\mathscr{D}_{X}[S]$-module $M$ is \emph{relative holonomic} if its relative characteristic variety satisfies
\[
\text{Ch}^{\text{rel}}(M) = \bigcup\limits_{\alpha} \Lambda_{\alpha} \times S_{\alpha} \subseteq T^{\star}X \times \mathbb{C}^{r},
\]
where the $\Lambda_{\alpha}$ are irreducible, conical Lagrangian varieties in the cotangent bundle $T^{\star}X$ and the $S_{\alpha}$ are irreducible algebraic varieties in $\mathbb{C}^{r}$. We denote by $p_{2}(\text{Ch}^{\text{rel}}(M))$ the image of the relative characterstic variety under the canonical projection $p_{2}: T^{\star}X \times \mathbb{C}^{r} \to \mathbb{C}^{r}.$
\end{define}

Maisonobe demonstrated in \cite{Maisonobe} that $\mathscr{D}_{X, \xpoint}[S]F^{S} / \mathscr{D}_{X, \xpoint}[S]F^{S+\textbf{1}}$ has similar properties as in classical $F = (f)$ setting, cf. Resultat 2, Resultat 3, and Proposition 9 of loc. cit. Most of these properties were generalized in \cite{ZeroLociI, ZeroLociII}. We summarize the ones we need below. Note that (iv) requires interpreting $B_{F, \mathfrak{x}}^{\textbf{a}}$ as the $\mathbb{C}[S]$-annihilator of $\mathscr{D}_{X, \xpoint}[S] F^{S} / \mathscr{D}_{X, \xpoint}[S]F^{S+\textbf{a}}$. \iffalse  (See \cite{ZeroLociI}, in particular Theorem 3.5.1, for an English summary.) \fi

\begin{theorem} \text{ \normalfont(Theorem 3.2.1 \cite{ZeroLociII}; Lemma 3.4.1 \cite{ZeroLociI})} \label{thm-ZeroLociII results}
Suppose that $f^{\textbf{a}}$ is not a unit. Then the following are true for the $\mathscr{D}_{X, \xpoint}[S]$-module $\mathscr{D}_{X, \xpoint}[S] F^{S} / \mathscr{D}_{X, \xpoint}[S]F^{S+\textbf{a}}$:
\begin{enumerate}[(i)]
    \item it is relative holonomic;
    \item it has grade $n+1$; 
    \item $\dim \text{\normalfont Ch}^{\text{\normalfont rel}}(\mathscr{D}_{X, \xpoint}[S] F^{S} / \mathscr{D}_{X, \xpoint}[S] F^{S+1}) = n + r - 1$; 
    \item $Z(B_{F, \xpoint}^{\textbf{a}}) = p_{2} (\Ch^{\rel}(\mathscr{D}_{X, \xpoint}[S] F^{S} / \mathscr{D}_{X, \xpoint}[S] F^{S+\textbf{a}})). $
\end{enumerate}
\end{theorem}

\iffalse
\begin{proposition} \text{\normalfont Compare to Theorem 3.2.2 of \cite{ZeroLociI}}
Suppose $N$ is a finitely generated $\mathscr{D}_{X}[S]$-module. Then 
\begin{enumerate}[(i)]
    \item $j(M) + \dim(\text{Ch}^{\text{rel}}(M) = 2n + r$;
    \item $j(M) + \dim(\text{Ch}^{\text{tot}}(M) = 2n + r$;
\end{enumerate}
\end{proposition}
\begin{proof}
For either filtration on $\mathscr{D}_{X}[S]$, argue as in Theorem 3.2.2 of \cite{ZeroLociI}.
\end{proof}
\fi

\subsection{Hypotheses on the Logarithmic Data of $\text{Div}(f)$} \text{ }

Here we introduce geometric conditions on the $\text{Div}(f)$ that were first considered in \cite{uli} and later in \cite{Bath1}. All of them involve the logarithmic information of $f$ and do not depend on the choice of defining equation of $f$. We generally do not assume $\text{Div}(f)$ is reduced and, as such (and following \cite{uli}), differentiate between the logarithmic data of $\text{Div}(f)$ and of $\text{Div}(f_{\text{red}})$. Most of our hypotheses hold irrespective of a reduced assumption, cf. Proposition \ref{prop - appendix relating log forms along D to reduced D}, Corollary \ref{cor - appendix, listing similar properties for non-reduced} as well as Appendix A at large. Nevertheless one should take heed of the terminal item of Remark \ref{rmk-strongly Euler}. \iffalse (see the terminal items of Remark \ref{rmk-Saito-holonomic}, Remark \ref{rmk-log ddifferential forms}, and Remark \ref{rmk-strongly Euler}--the last is potentially disruptive.) \fi

\iffalse
Here we will introduce conditions on the divisor cut out by $f$ (where $f$ is equipped with some factorization $f = f_{1} \cdots f_{r}$) that guarantee that $\mathscr{D}_{X,\xpoint}[S]F^{S} / \mathscr{D}_{X,\xpoint}[S]F^{S}$ is $(n+1)$-Cohen--Macaulay. This will grant the Bernstein--Sato ideal of $F$ at $\xpoint$ very nice properties. These hypotheses on $\text{Div}(f)$ were first studied in \cite{uli} in the univariate setting; in \cite{Bath1} we showed that in the multivariate setting these properties still have very good consequences, cf. Theorem \ref{thm- gen by derivations}. 
\fi

\begin{define}

Let $\mathscr{I}_{\text{Div}(f)} \subseteq \mathscr{O}_{X}$ be the ideal sheaf of $\text{Div}(f)$ and $\Der_{X}$ the sheaf of derivations on $X$. The \emph{logarithmic derivations} are the subsheaf $\Der_{X}(-\log f)$ of $\Der_{X}$ locally generated by the derivations $\delta$ such that $\delta \bullet (\mathscr{I}_{\text{Div}(f)}) \subseteq \mathscr{I}_{\text{Div}(f)}$.

\end{define}

Straightforward computations with the product rule imply that $\Der_{X}(-\log f) = \Der_{X}(-\log f_{\text{red}})$ and that the logarithmic derivations do not depend on the choice of defining equation. However, direct sum decompositions of the logarithmic derivations ala Remark 2.10 of \cite{uli} may depend on reducedness and/or the choice of defining equation. The following hypotheses helps to alleviate some of the issues involving choice of equation, cf. Remark 3.2 of loc. cit. and Remark 2.15 of \cite{Bath1}.

\begin{define} \label{def- strongly Euler-homogeneous}
We say $f$ is strongly Euler-homogeneous at $\xpoint$ if there is a derivation $E_{\xpoint} \in \Der_{X,\xpoint}(-\log f)$ such that (1) $E_{\xpoint} \bullet f = f$ and (2) $E_{x}$ vanishes at $\xpoint$. We say $f$ is strongly Euler-homogeneous if it is strongly Euler-homogeneous for all $\mathfrak{x} \in \text{V}(f)$.
\end{define}

\begin{remark} \label{rmk-strongly Euler}
\begin{enumerate}[(a)]
    \item The condition ``strongly Euler-homogeneous'' does not depend on the choice of local defining equation of $\text{Div}(f)$: if $E_{\xpoint} \bullet f = f$, then $\frac{u E_{\xpoint}}{u + E_{\xpoint} \bullet u} \bullet u f = uf.$ If you remove condition (2), the choice may matter.
    \item Hyperplane arrangements are strongly Euler-homogeneous. At the origin, the Euler derivation $\sum x_{i} \partial_{i}$ is a strong Euler-homogeneity; an appropriate coordinate change deals with the other points. Divisors that locally everywhere admit a choice of coordinates so that they can be defined by a ``homogeneous polynomial'' (with respect to an appropriate weight system) are known as \emph{quasi-homogeneous divisors} and are also strongly Euler-homogeneous.
     \item If $m \in \mathbb{Z}_{\geq 1}$, then $f$ is strongly Euler-homogeneous at $\mathfrak{x}$ if and only if $f^{m}$ is strongly Euler-homogeneous at $\mathfrak{x}$.
    \item Whether or not ``$f^{\textbf{a}}$ is strongly Euler-homogeneous at $\mathfrak{x}$'' implies ``$f^{\textbf{b}}$ is strongly Euler-homogeneous at $\mathfrak{x}$'' for arbitrary $\textbf{a}, \textbf{b} \in \mathbb{Z}_{\geq 1}^{r}$  seems quite subtle. It is not known to us, even for quasi-homogeneous divisors. When $f$ is a hyperplane arrangement this is true: up to coordinate change and a scaling, the Euler derivation is the requisite homogeneity. Essentially, this is because all the linear factors are homogeneous with respect to the same weight system. However, it seems possible that $f$ may be quasi-homogeneous without all of its factors $f_{k}$ being quasi-homogeneous; or, each $f_{k}$ may be quasi-homogeneous with respect to different weight systems that somehow assemble into a quasi-homogeneous $f$ with respect to a new weight system. Geometrically one might expect the aforementioned implication to hold at least for quasi-homogeneous $f$ since it can be visualized as a $\mathbb{C}^{\star}$-equivariant action on $\Variety(f)$.
\iffalse
    \item Consider $f = z(x^{5} + y^{5} + zx^{2}y^{3})$ from Section 4.2 of \cite{BahloulOakuLocalBSIdeals}. Then $x \partial_{x} + y \partial_{y}$ is an Euler-homogeneity and using Macaulay2 we may find a basis $\{a_{i} \partial_{x} + b_{i} \partial_{y} + c_{i} \partial_{z} \}_{i=1}^{4}$
\fi
\end{enumerate}
\end{remark}

In \cite{SaitoLogarithmicForms}, Saito used the logarithmic derivations to stratify $X$. Indeed, define an equivalence relation on $X$ as follows: two points $\mathfrak{x}$ and $\mathfrak{y}$ are related if there is an open set $U \subseteq X$ and a logarithmic derivation $\delta \in \Der_{U}(-\log f)$ that (1) does not vanish on $U$ and (2) the integral curve of $\delta$ passes through $\mathfrak{x}$ and $\mathfrak{y}$. The resulting stratification is the \emph{logarithmic stratification} and we call the strata the \emph{logarithmic strata}. \iffalse While Saito assumed reducedness, since $\Der_{X}(-\log f) = \Der_{X}(-\log f_{\text{red}})$ this construction can be thought of as only relying on the reduced structure. \fi

\begin{define} \label{def-Saito holonomic}
We say $f$ is \emph{Saito-holonomic} if the logarithmic stratification of $X$ by $\Der_{X}(-\log f)$ is locally finite. 
\end{define}

\begin{remark} \label{rmk-Saito-holonomic}
\begin{enumerate}[(a)]
    \item As $\Der_{X}(-\log f) = \Der_{X}(-\log f_{\red})$, the logarithmic stratification depends only on the reduced structure of $f$.
    \item The $n$-dimensional logarithmic strata are the connected components of $X \setminus \Variety(f)$ and the $(n-1)$-dimensional strata are the smooth connected components of $\Variety(f_{\red}).$
    \item Hyperplane arrangements are Saito-holonomic, cf. Example 3.14 of \cite{SaitoLogarithmicForms}.
    \item Consider $f = z(x^{5} + y^{5} + zx^{2}y^{3})$ from Section 4.2 of \cite{BahloulOakuLocalBSIdeals}. This is not Saito-holonomic: every point on the $z$-axis is a zero dimensional logarithmic stratum.
    
\end{enumerate}
\end{remark}

\iffalse
\begin{define} 
We define an equivalence relation on $X$ as follows. Two points $\xpoint, \mathfrak{y} \in X$ are related if there is an open set $U \subseteq X$ and a logarithmic derivation $\delta \in \Der_{U}(-\log f)$ that does not vanish on $U$ and whose integral curve passes through $\xpoint$ and $\mathfrak{y}.$ The irreducible components of this equivalence relation constitute the \emph{logarithmic strata} of $f$, or $\text{Div}(f)$; we say $f$, or $\text{Div}(f)$, is \emph{Saito-holonomic} if this stratification is locally finite. 
\end{define}
\fi

In \cite{SaitoLogarithmicForms}, Saito also introduced (originally under a reducedness hypothesis) a sort of dual object to the logarithmic derivations:

\begin{define} \label{def-log forms}
Consider the classical sheaf of differential forms on $X$, $\Omega_{X}^{\bullet}$, whose differential $d$ is exterior differentiation. The sub-sheaf of \emph{logarithmic $k$-forms} $\Omega_{X}^{k}(\log f)$ satisfy
\[
\Omega_{X}^{k}(\log f) := \{ \omega \in \frac{1}{f} \Omega_{X}^{k} \mid d(\omega) \in \frac{1}{f} \Omega_{X}^{k+1}. \}
\]
We say $f$ is \emph{tame} at $\xpoint$ if the projective dimension of the $\mathscr{O}_{X,\xpoint}$-module $\Omega_{X, \xpoint}^{k}(\log f)$ is at most  $k$ for all $0 \leq k \leq n$; moreover, $f$ is simply tame if it so at each $\xpoint \in X.$ We say $f$ is \emph{free} at $\xpoint$ if $\Omega_{X,\xpoint}^{1}(\log f)$ is a free $\mathscr{O}_{X,\xpoint}$-module; $f$ is simply free if it is so at each $\xpoint \in X.$ 
\end{define}

While Saito (and others) originally considered the logarithmic $k$-forms, freeness, and tameness only for reduced divisors, we do not make this restriction. To our knowledge, \cite{uli} is one of the first places where these objects were defined in the non-reduced setting. In Appendix A we catalogue the basic properties of $\Omega_{X}^{k}(\log f)$ for non-reduced $f$ and show that, up to $\mathscr{O}_{X}$-module isomorphisms, the non-reduced logarithmic $k$-forms can be identified with the reduced logarithmic $k$-forms, cf. Proposition \ref{prop - appendix relating log forms along D to reduced D}.

\begin{remark} \label{rmk-log ddifferential forms}
\begin{enumerate}[(a)]
    \item The condition of freeness is stronger than tameness. By 1.7, 1.8 of \cite{SaitoLogarithmicForms}, if $f_{\red}$ is free, then $\Omega_{X,\xpoint}^{k}(\log f_{\red}) \simeq \bigwedge^{k} \Omega_{X,\xpoint}^{1}(\log f_{\red}).$ (Saito's originally argument assumes reducedness, but his argument for this item does not depend on this point. Also see item (c).)
    \item If $\dim X \leq 3$ then any $f \in \mathscr{O}_{X}$ is automatically tame. If $\dim X = 4$ then tameness is equivalent to $\text{proj dim } \Omega_{X,\xpoint}^{1}(\log f) \leq 1.$ These facts follow from reflexivity of the logarithmic differential forms, cf. 1.7 of \cite{SaitoLogarithmicForms} for one case. Tameness and freeness only depend on the reduced structure of the divisor cut out by $f$, cf. Corollary \ref{cor - appendix, listing similar properties for non-reduced}.
\end{enumerate}
\end{remark}

\section{Results on Multivariate Bernstein--Sato Ideals}

We will now assume $\text{Div}(f)$ is strongly Euler-homogeneous, Saito-holonomic and tame. In this section we prove the first three theorems from the Introduction. The general strategy is to pass the properties proved for $\gr^{\sharp}(\mathscr{D}_{X,\xpoint}[S]F^{S})$ in \cite{Bath1} to $\mathscr{D}_{X,\xpoint}[S]F^{S} / \mathscr{D}_{X,\xpoint}[S] F^{S+\textbf{a}}$ and then to use the definition of relative holonomic and Theorem \ref{thm-ZeroLociII results} to pass these properties to $\gr^{\rel}(\mathscr{D}_{X,\xpoint}[S]F^{S} / \mathscr{D}_{X,\xpoint}[S]F^{S+\textbf{a}})$. \iffalse The main tool is the technical lemma Proposition 4.3.4 of \cite{ZeroLociI} and the fact that, under these assumptions, $\mathscr{D}_{X,\xpoint}[S]F^{S} / \mathscr{D}_{X,\xpoint}[S]F^{S+\textbf{a}}$ is $(n+1)$-Cohen--Macaulay, cf. Theorem \ref{thm- F^S / F^S+1 CM}. \fi

\subsection{Criteria for $\mathscr{D}_{X,\xpoint}[S]F^{S} / \mathscr{D}_{X,\xpoint}[S]F^{S+\textbf{a}}$ to be $(n+1)$-Cohen--Macaulay} \text{ }

In general, elements of $\mathscr{D}_{X}[S]$ that annihilate $F^{S}$ can be very complicated. The only such elements that are easy to identify are essentially encoded by the logarithmic derivations:

\begin{define} \label{def-gen by derivations}
There is an $\mathscr{O}_{X}[S]$-linear map $\psi_{F}: \Der_{X}(-\log f) \to \ann_{\mathscr{D}_{X}[S]} F^{S}$
\[
\psi_{F}(\delta) = \delta - \sum s_{k} \frac{\delta \bullet f_{k}}{f_{k}}.
\]
We say $\ann_{\mathscr{D}_{X}[S]}F^{S}$ is \emph{generated by derivations} if
\[
\ann_{\mathscr{D}_{X}[S]} F^{S} = \mathscr{D}_{X}[S] \cdot \psi_{F}(\Der_{X}(-\log f))
\]
\end{define}

By Theorem 2.29 of \cite{Bath1}, if $f$ is strongly Euler-homogeneous, Saito-holonomic, and tame, then $\ann_{\mathscr{D}_{X,\xpoint}}[S] F^{S}$ is generated by derivations. Moreover, by loc. cit., the associated graded object attached to the total order filtration has good properties:

\begin{theorem} \text{\normalfont (Theorem 2.23, Corollary 2.28, Theorem 2.29 \cite{Bath1})} \label{thm- gen by derivations}
Suppose that $f$ is strongly Euler-homogenous, Saito-holonomic, and tame and $F$ corresponds to any factorization $f = f_{1} \cdots f_{r}$. Then $\ann_{\mathscr{D}_{X}[S]} F^{S}$ is generated by derivations and $\gr^{\sharp}(\ann_{\mathscr{D}_{X}[S]} F^{S})$ is locally a Cohen--Macaulay, prime ideal of dimension $n+r.$
\end{theorem}

The Cohen--Macaulay properties for $\gr^{\sharp}(\mathscr{D}_{X,\xpoint}[S]F^{S})$ descend to $\mathscr{D}_{X,\xpoint}[S]F^{S}:$

\begin{proposition} \label{prop- F^S CM}
Suppose that $f$ is strongly Euler-homogenous, Saito-holonomic, and tame and $F$ corresponds to any factorization $f = f_{1} \cdots f_{r}$. Then \iffalse for all $\xpoint \in X$, \fi $\mathscr{D}_{X, \xpoint}[S] F^{S}$ is $n$-Cohen--Macaulay.
\end{proposition}

\begin{proof}
We first show that the hypotheses imply $\gr^{\sharp}(\mathscr{D}_{X, \xpoint}[S]) / \gr^{\sharp}(\ann_{\mathscr{D}_{X, \xpoint}[S]} F^{S})$ is a $n$-Cohen--Macaulay $\gr^{\sharp}(\mathscr{D}_{X, \xpoint}[S])$-module. \iffalse Because $\gr_{(0,1,1)}(\mathscr{D}_{X}[S]) / \gr_{(0,1,1)}(\ann_{\mathscr{D}_{X, \xpoint}[S]} F^{S}))$ \fi Since this module is a graded local $\gr^{\sharp}(\mathscr{D}_{X,\xpoint}[S])$-module (with respect to the grading induced by the total order filtration), the vanishing of $\Ext$-modules is a graded local condition. (See Proposition 1.5.15(c) of \cite{BrunsHerzog} where faithful exactness of localizations at graded maximal ideals is proved.) Thus, it suffices to show $\gr^{\sharp}(\mathscr{D}_{X, \xpoint}[S]) / \gr^{\sharp}(\ann_{\mathscr{D}_{X, \xpoint}[S]} F^{S})$ is a $n$-Cohen--Macaulay module after localizing at the graded maximal ideal. This follows from Theorem \ref{thm- gen by derivations} and a routine commutative algebra argument we summarize below.

Suppose $R$ is a commutative, Cohen--Macaulay, Noetherian, local ring and $M$ a finitely generated Cohen--Macaulay $R$-module of dimension $\dim M$. We claim $M$ is $(\text{codim } M)$-Cohen--Macaualy. By Auslander-Buchsbaum and since $M$ is Cohen--Macaulay, we know that $\text{dim} M = \text{depth} M$ = $\text{depth } R - \text{proj dim } M.$ So $\text{proj dim} M = \text{codim } M. $ Since $R$ is local, Cohen-Macaulay and $M$ is finitely generated, $\text{height} \ann M + \dim M = \dim R.$ Thus $\text{proj dim} M = \text{height} \ann M = \text{codim } M$. Let $j(M)$ be the grade of $M$. Then to show $M$ is $(\text{codim} M)$--Cohen--Macaulay it suffices to show $j(M) = \text{proj dim } M$. \iffalse and the latter equals $\ell$ since $R$ is local and Cohen-Macaulay. Let $j(M)$ be the grade of $M$. We will show $\dim R - p = k$ completing the proof of the proposition. \fi Certainly $j(M) \leq \text{proj dim } M$.  On the other hand, let $\ell$ be the length of a maximal regular $R$-sequence in $\ann M$. Then $0 \to R \to R \to R / R \cdot x \to 0$ induces the short exact sequence $0 \to \Ext_{R}^{i-1}(M, R) \to \Ext_{R}^{i-1}(M, R / R \cdot x) \to \Ext_{R}^{i}(M, R) \to 0.$ So if $\Ext_{R}^{i}(M, R) \neq 0$, then $\Ext_{R}^{i-1}(M, R / R \cdot x) \neq 0.$ Iterating this procedure, we see that \iffalse This lets us do an induction argument on $\ell$ to show that \fi $j(M) \geq \ell. $ Since $R$ is local and Cohen--Macaulay, $\ell = \text{height} \ann M$ and the claim is proved. \iffalse By Auslander-Buchsbaum and the fact $M$ is Cohen--Macaulay, $\text{dim} M = \text{depth} M$ = $\text{depth } R - \text{proj dim } M.$ So $\text{proj dim} M = \text{codim } M. $ Since $R$ is local and $M$ is Cohen-Macaulay, $\text{height} \ann M + \dim M = \dim R.$ Thus $\text{proj dim} M = \text{height} \ann M$ and the latter equals $\ell$ since $R$ is local and Cohen-Macaulay. \fi

So the above argument and Theorem \ref{thm- gen by derivations} imply $\gr^{\sharp}(\mathscr{D}_{X, \xpoint}[S]) / \gr^{\sharp}(\ann_{\mathscr{D}_{X,\xpoint}}[S]F^{S})$ is $n$-Cohen--Macaulay. Since $\mathscr{D}_{X,\xpoint}[S]F^{S}$ is cyclically generated by $F^{S}$, we have the obvious isomorphism $\gr^{\sharp}(\mathscr{D}_{X,\xpoint}[S] F^{S}) \simeq \gr^{\sharp}(\mathscr{D}_{X, \xpoint}[S]) / \gr^{\sharp}(\ann_{\mathscr{D}_{X,\xpoint}}[S]F^{S})$. Recall that we can apply the lemmas in subsection 2.1 to $\mathscr{D}_{X,\xpoint}[S]$ and the total order filtration. By Lemma \ref{lemma-good filtration on ext} we see that there, for each $k$, there is a good filtration $\Gamma$ on $\Ext_{\mathscr{D}_{X,\xpoint}[S]}^{k}(\mathscr{D}_{X,\xpoint}[S]F^{S}, \mathscr{D}_{X,\xpoint}[S])$ such that 
\[
\gr^{\Gamma}( \Ext_{\mathscr{D}_{X,\xpoint}[S]}^{k}(\mathscr{D}_{X,\xpoint}[S]F^{S}, \mathscr{D}_{X,\xpoint}[S]))
\]
is a subquotient of
\[
\Ext_{\gr^{\sharp}(\mathscr{D}_{X,\xpoint}[S])}^{k}(\gr^{\sharp}(\mathscr{D}_{X,\xpoint}[S]F^{S}), \gr^{\sharp}(\mathscr{D}_{X,\xpoint}[S])).
\]
Since $\mathscr{D}_{X,\xpoint}[S]$ is positively filtered by the total order filtration, for any good filtration on a finitely generated, left $\mathscr{D}_{X,\xpoint}[S]$-module $M$, we know that if $M$ is nonzero then $\gr M$ is nonzero. Therefore
\[
\Ext_{\mathscr{D}_{X,\xpoint}[S]}^{k}(\mathscr{D}_{X,\xpoint}[S]F^{S}, \mathscr{D}_{X,\xpoint}[S]) = 0 \text{ for } k \neq n.
\]
That 
\[
\Ext_{\mathscr{D}_{X,\xpoint}[S]}^{n}(\mathscr{D}_{X,\xpoint}[S]F^{S}, \mathscr{D}_{X,\xpoint}[S]) \neq 0
\]
follows from, for instance, Proposition \ref{prop- M pure grade is associated graded grade} and our computation of the grade of $\gr^{\sharp}(\mathscr{D}_{X,\xpoint}[S] F^{S})$ above. 
\end{proof}

\iffalse 

\begin{remark} \label{rmk- F^S+1 CM}
It is shown in the proof of Theorem 2.20 of \cite{Bath2} that 
\[
\gr_{(0,1,1)}(\ann_{\mathscr{D}_{X}[S]} F^{S}) = \gr_{(0,1,1)}(\ann_{\mathscr{D}_{X}[S]} F^{S+1}).
\]
(See equation (2.1) of loc cit.) The argument of Proposition \ref{prop- F^S CM} thus applies to $\mathscr{D}_{X, \xpoint}[S]F^{S+1}$ and we conclude $\mathscr{D}_{X,\xpoint}[S]F^{S+1}$ is $n$-Cohen--Macaulay.
\end{remark}
\fi

The Cohen--Macaulay property descends further to $\mathscr{D}_{X}[S]F^{S} / \mathscr{D}_{X}[S] F^{S+\textbf{a}}$:

\begin{theorem} \label{thm- F^S / F^S+1 CM}
Suppose that $f$ is strongly Euler-homogenous, Saito-holonomic, and tame, $F$ corresponds to any factorization $f = f_{1} \cdots f_{r}$, and $f^{\textbf{a}}$ is not a unit. Then \iffalse for all $\xpoint \in X$, \fi
\[
\frac{\mathscr{D}_{X, \xpoint}[S] F^{S}}{\mathscr{D}_{X, \xpoint}[S]F^{S+\textbf{a}}} \text{ is $(n+1)$-Cohen--Macaulay.}
\]
\end{theorem}

\begin{proof} 
We can generalize Proposition \ref{prop- F^S CM} and show that $\mathscr{D}_{X,\xpoint}[S]F^{S+\textbf{a}}$ is $n$-Cohen--Macaulay even if $\textbf{a} \neq \textbf{0}$. In the language of \cite{Bath2}, $\mathscr{D}_{X,\xpoint}[S]F^{S+\textbf{a}} = \mathscr{D}_{X,\xpoint}[S] f^{\textbf{a}} F^{S}$ where $f^{\textbf{a}} = f_{1}^{a_{1}} \cdots f_{r}^{a_{r}}$ is certainly compatible with respect to $f$, cf. Definition 2.14 of loc. cit. Thus we may use the results of Section 2 of \cite{Bath2}. In particular, combining Theorem 2.20 and Remark 2.19 of \cite{Bath2} along with Corollary 2.28 of \cite{Bath1} we deduce
\[
\gr^{\sharp} (\ann_{\mathscr{D}_{X,\xpoint}[S]} F^{S}) = \gr^{\sharp} (\ann_{\mathscr{D}_{X,\xpoint}[S]} F^{S + \textbf{a}}).
\]
So we may repeat the argument of Proposition \ref{prop- F^S CM} for $\mathscr{D}_{X,\xpoint}[S]F^{S + \textbf{a}}$ and deduce this module is $n$-Cohen--Macaulay. (Note that Proposition \ref{prop- M pure grade is associated graded grade} also shows the grade of $\mathscr{D}_{X,\mathfrak{x}}[S] F^{S+\textbf{a}}$ equals the grade of $\gr^{\sharp}(\mathscr{D}_{X,\xpoint}[S] F^{S+\textbf{a}})$ so the final sentence of the proof of Proposition \ref{prop- F^S CM} also holds in this setting.)

Now consider short exact sequence of $\mathscr{D}_{X, \xpoint}[S]$-modules:
\[
0 \xrightarrow{} \mathscr{D}_{X,\xpoint}[S]F^{S+\textbf{a}} \xrightarrow{} \mathscr{D}_{X,\xpoint}[S]F^{S} \xrightarrow{} \frac{\mathscr{D}_{X, \xpoint}[S] F^{S}}{\mathscr{D}_{X, \xpoint}[S]F^{S+\textbf{a}}} \xrightarrow{} 0.
\]
Apply $\Hom_{\mathscr{D}_{X,\xpoint}[S]}(-, \mathscr{D}_{X,\xpoint}[S])$ and consider the canonical long exact sequence of $\Ext$ modules. Because $\mathscr{D}_{X,\xpoint}[S]F^{S}$ and $\mathscr{D}_{X,\xpoint}[S]F^{S+\textbf{a}}$ are both $n$-Cohen--Macaulay, clearly
\[
\Ext_{\mathscr{D}_{X,\xpoint}[S]}^{k}(\mathscr{D}_{X, \xpoint}[S]F^{S} / \mathscr{D}_{X,\xpoint}[S]F^{S+\textbf{a}}, \mathscr{D}_{X,\xpoint}[S]) = 0 \text{ for all } k \neq n, n+1.
\]
By Theorem \ref{thm-ZeroLociII results}, the grade of $\mathscr{D}_{X, \xpoint}[S]F^{S} / \mathscr{D}_{X, \xpoint}[S]F^{S+\textbf{a}}$ is $n+1$.
\end{proof}

\iffalse
\begin{remark}
We can also determine $j(\mathscr{D}_{X, \xpoint}[S]F^{S} / \mathscr{D}_{X, \xpoint}[S]F^{S+\textbf{a}} > n$ (which is all we need above) by relative holonomicity, the fact $B_{F,0}^{\textbf{1}} \neq 0$, and Theorem \ref{thm-ZeroLociII results} (iv). These are simpler results than the grade computation of Theorem \ref{thm-ZeroLociII results}(ii).
\end{remark}
\fi

\subsection{Applications to the Bernstein--Sato Ideal} \text{ }

It behooves us to study a larger class of $\mathbb{C}[S]$-annihilators of finite $\mathscr{D}_{X,\mathfrak{x}}[S]$-modules.

\begin{define} \label{def-BS ideal general}
For a finite $\mathscr{D}_{X,\mathfrak{x}}[S]$-module $M$, its \emph{Bernstein--Sato ideal} is
\[
B_{M} = \ann_{\mathbb{C}[S]} M.
\]
Note that, as one would hope, $B_{F,\mathfrak{x}}^{\textbf{a}}$ is shorthand for the Bernstein--Sato ideal of $\mathscr{D}_{X,\mathfrak{x}}[S]F^{S} / \mathscr{D}_{X,\mathfrak{x}}[S]F^{S+\textbf{a}}$.
\end{define}

While our primary interest is $\mathscr{D}_{X,\mathfrak{x}}[S]F^{S} / \mathscr{D}_{X,\mathfrak{x}}[S]F^{S+\textbf{a}}$, the crucial property Theorem \ref{thm-ZeroLociII results}.(iv) about Bernstein--Sato ideals applies more generally: if $M$ is relative holonomic, and with $p_{2}: T^{\star}X \times \Spec \mathbb{C}[S] \to \Spec \mathbb{C}[S]$ the canonical projection, then $Z(B_{M}) = p_{2}(\Ch^{\rel} M)$, cf. Lemma 3.4.1 of \cite{ZeroLociI}. Also note that if $N \subseteq M$ are finite $\mathscr{D}_{X,\mathfrak{x}}[S]$-modules such that $M$ is relative holonomic, then relative holonomicity descends to $N$, cf. Lemma 3.2.2 and Proposition 3.2.5 of \cite{ZeroLociI}. (The essence of the latter proposition originally appears in \cite{Maisonobe} using Maisonobe's terminology \emph{major\'e par une lagrangienne}). 

One difficulty is that the first item of Proposition \ref{prop- pure iff associated prime same heights} does not apply for a polynomial ring over the stalk of the analytic structure sheaf $\mathscr{O}_{X,\mathfrak{x}}$: there are maximal ideals of different heights. Nevertheless, Maisonobe showed the desired form of Proposition \ref{prop- pure iff associated prime same heights}.(1) still does hold. We state a useful version combining $\mathscr{D}_{X,\mathfrak{x}}[S]$-data and associated graded data--recall grade can be computed on either side, cf. Proposition \ref{prop- M pure grade is associated graded grade}. Namely:
if $M$ is a finite $\mathscr{D}_{X,\mathfrak{x}}[S]$-module then
\begin{equation} \label{eqn- local analytic dimension grade fml}
j(M) + \dim (\Ch^{\rel} M) = 2n + r. 
\end{equation}
(Recall that $\mathscr{D}_{X,\mathfrak{x}}[S] = \mathscr{D}_{X,\mathfrak{x}}[s_{1}, \dots, s_{r}]$.) Maisonobe's reasoning is succinctly described in more generality, and in English, in section 3.6 of \cite{ZeroLociI} where this formula, as well as the philosophy for turning algebraic results into local analytic ones, is detailed.

Armed with these facts we can give criterion for a Bernstein--Sato ideal to be principal. The proof non-trivially builds on (and reveals the versatility of) an idea of Maisonobe's originating in Theorem 2 of \cite{MaisonobeFreeHyperplanes}, that we also used in Proposition 3.13 of \cite{Bath2}.

\begin{theorem} \label{thm- purity implies principality} Suppose that $M$ is a finite $\mathscr{D}_{X,\mathfrak{x}}[S]$-module that is relative holonomic and $(n+1)$-pure. Then its Bernstein--Sato ideal $B_{M}$ is principal.
\end{theorem}

\begin{proof}
First we prove that $Z(B_{M})$ is purely codimension one. For this argument we only consider the relative order filtration on $\mathscr{D}_{X,\mathfrak{x}}[S]$. By Lemma \ref{lemma-criterion for purity}, there is a good filtration on $M$ such that the associated graded module is a pure $\gr^{\rel}(\mathscr{D}_{X,\xpoint}[S])$-module of grade $n+1$. Since the relative characteristic variety of any finitely generated $\mathscr{D}_{X,\xpoint}[S]$ module does not depend on the choice of good filtration, Proposition \ref{prop- pure iff associated prime same heights}.(2) and equation \eqref{eqn- local analytic dimension grade fml} imply $\text{Ch}^{\text{rel}}(\mathscr{D}_{X,\xpoint}[S] F^{S} / \mathscr{D}_{X,\xpoint}[S]F^{S+1})$ is equidimensional of dimension $n+r-1$. By assumption $M$ is relative holonomic, so in the product structure of Definition \ref{def- relative holonomic} all the irreducible varieties $S_{\alpha} \subseteq \mathbb{C}^{r}$ have dimension $r-1$. And by the properties of the projection $p_{2}: T^{\star}X \times \Spec \mathbb{C}[S] \to \Spec \mathbb{C}[S]$ outlined above (compare to Theorem \ref{thm-ZeroLociII results})
\[
Z(B_{M}) = p_{2}(\Ch^{\rel}(B_{M}) = \bigcup_{\alpha} S_{\alpha}.
\]
So $Z(B_{M})$ is purely codimension one. 

Now we show the principality of $\rad (B_{M})$ implies the principality of $B_{M}$. Pick a generating set $(b_{1}, \dots, b_{t})$ of $B_{F,\mathfrak{x}}^{\textbf{a}}$. As $\rad (B_{F,\mathfrak{x}})$ is principal, we may write each $b_{i} = \beta_{i} \alpha_{i}$ so that (1) $\rad (\mathbb{C}[S] \cdot \beta_{i}) = \rad (B_{F,\mathfrak{x}}^{a})$ and (2) $Z(\alpha_{i}) \cap Z(B_{F,\mathfrak{x}}^{\textbf{a}})$ has codimension at least $2$. (This is possible since $\mathbb{C}[S]$ is a UFD: factor $b_{i}$ into irreducibles and group all the irreducible factors (and their powers) into those whose reduced form divide $\rad (B_{M})$ and those who do not: the former constitute $\beta_{i}$; the latter $\alpha_{i}$.) Now consider
\[
N_{i} = \beta_{i} \frac{\mathscr{D}_{X,\mathfrak{x}}[S]F^{S}}{ \mathscr{D}_{X,\mathfrak{x}}[S]F^{S+\textbf{a}}} \subseteq \frac{\mathscr{D}_{X,\mathfrak{x}}[S]F^{S}}{ \mathscr{D}_{X,\mathfrak{x}}[S]F^{S+\textbf{a}}}.
\]
As $b_{i} \in B_{F,\mathfrak{x}}^{\textbf{a}}$, certainly $\alpha_{i}$ kills $N_{i}$. On the other hand, since $N_{i}$ is a submodule of $\mathscr{D}_{X,\mathfrak{x}}[S]F^{S} / \mathscr{D}_{X,\mathfrak{x}}[S]F^{S+\textbf{a}}$, we also have the containment $Z(\ann_{\mathbb{C}[S]}N_{i}) \subseteq Z(B_{F,\mathfrak{x}})$. Together, $Z(\ann_{\mathbb{C}[S]}N_{i}) \subseteq Z(B_{F,\mathfrak{x}}) \cap Z(\alpha_{i})$, which has codimension at least $2$. But $N_{i}$ inherits $(n+1)$-purity from $\mathscr{D}_{X,\mathfrak{x}}[S]F^{S} / \mathscr{D}_{X,\mathfrak{x}}[S]F^{S+\textbf{a}}$, which, from the prelude, implies that: if $N_{i}$ is nonzero, $Z(\ann_{\mathbb{C}[S]} N_{i})$ is purely codimension one. So we deduce $N_{i} = 0$. 

The upshot is that the generators $(b_{1}, \dots, b_{t})$ of $B_{F,\mathfrak{x}}$ satisfies $\rad (\mathbb{C}[S] \cdot b_{i}) = \rad (B_{F,\mathfrak{x}}^{\textbf{a}})$ for all $i$. Now write $d = \text{gcd}(b_{1}, \dots, b_{t})$ and pick the ideal $I \subseteq \mathbb{C}[S]$ so that $(\mathbb{C}[S] \cdot d) \cdot I = B_{F,\mathfrak{x}}^{\textbf{a}}$. That is, $I = (B_{F,\mathfrak{x}}^{\textbf{a}} :_{\mathbb{C}[S]} \mathbb{C}[S] \cdot d).$ We will be done if we show $d$ kills $\mathscr{D}_{X,\mathfrak{x}}[S]F^{S} / \mathscr{D}_{X,\mathfrak{x}}[S]F^{S+\textbf{a}}$, so consider 
\[
N = d \cdot \frac{\mathscr{D}_{X,\mathfrak{x}}[S]F^{S}}{ \mathscr{D}_{X,\mathfrak{x}}[S]F^{S+\textbf{a}}} \subseteq \frac{\mathscr{D}_{X,\mathfrak{x}}[S]F^{S}}{ \mathscr{D}_{X,\mathfrak{x}}[S]F^{S+\textbf{a}}}.
\]
Like before, $N$ inherits $(n+1)$-purity. By construction $I$ kills $N$, so $Z(\ann_{\mathbb{C}[S]} N) \subseteq Z(I)$. And by construction $Z(I)$ is at least codimension $2$ (a codimension one component implies the gcd is too small). Again by the prelude, the $(n+1)$-purity of $N$ implies that if $N$ is nonzero, then $Z(\ann_{\mathbb{C}[S]} N)$ is purely codimension one. So we conclude $N$ must be zero. Thus $\mathbb{C}[S] \cdot d = B_{F,\mathfrak{x}}^{\textbf{a}}$ and we are done.
\end{proof}

Under our main working hypotheses we quickly obtain:
\iffalse
show that, under the hypotheses of strongly Euler-homogeneous, Saito-holonomic, and tame, $\Variety(\sqrt{B_{F,\xpoint}})$ has many nice properties. In particular, we will compare the reduced locus of the Bernstein--Sato variety corresponding to two different factorizations of $f$.

First, we show that under these hypotheses $\Variety(\sqrt{B_{F,\xpoint}})$ is purely codimension one. In general, the codimension one part of $\Variety(\sqrt{B_{F,\xpoint}})$ is a particularly nice hyperplane arrangement, cf. \cite{Maisonobe} Resultat 2 and \cite{ZeroLociI} Theorem 1.5.1.
\fi

\begin{corollary} \label{cor- tame implies BS principal}
Suppose that $f$ is strongly Euler-homogenous, Saito-holonomic, and tame, $F$ corresponds to any factorization $f = f_{1} \cdots f_{r}$, and $f^{\textbf{a}}$ is not a unit. Then $B_{F,\xpoint}^{\textbf{a}}$ is principal.
\end{corollary}

\begin{proof}
Once we show $\mathscr{D}_{X,\mathfrak{x}}[S]F^{S} / \mathscr{D}_{X,\mathfrak{x}}[S]F^{S+\textbf{a}}$ is $(n+1)$-pure this follows from Theorem \ref{thm- purity implies principality}. But we know something stronger: Theorem \ref{thm- F^S / F^S+1 CM} says it is $(n+1)$-Cohen--Macaulay.
\end{proof}

\begin{remark} \label{rmk- extending pure implies principal}
\begin{enumerate}[(a)]
    \item Theorem \ref{thm- purity implies principality} only requires that $\mathbb{C}[S]$ is a UFD, so a similar, and more general, result holds for Bernstein--Sato ideals over $\mathscr{D}_{X} \otimes_{\mathbb{C}} R$ instead of $\mathscr{D}_{X}[S]$ provided $R$ is assumed to be a UFD in addition to the standard assumptions of: $R$ is a commutative, finitely generated $\mathbb{C}$-algebra that is an integral domain, cf. \cite{ZeroLociI}. That is, under these assumptions, the $R$-annihilator of a finite $(n+1)$-pure module over $\mathscr{D}_{X} \otimes_{\mathbb{C}} R$ will be principal. Working with $\mathscr{D}_{X} \otimes_{\mathbb{C}} R$, where $R$ is a localization of $\mathbb{C}[S]$ at an ideal cutting out a closed variety of $\Spec \mathbb{C}[S]$, was used very effectively in loc. cit.; the above results thus apply to that setting.
    \item One can also obtain similar results in the algebraic $D_{X}$-module setting as there Proposition \ref{prop- pure iff associated prime same heights}.(1) applies directly.
\end{enumerate}
\end{remark}

The module $\mathscr{D}_{X,\mathfrak{x}}[S]F^{S} / \mathscr{D}_{X,\mathfrak{x}}[S]F^{S+\textbf{a}}$ may not be $(n+1)$-pure and its corresponding Bernstein--Sato ideal may not be principal, as the following example from Bahloul and Oaku demonstrates (see also \cite{BrianconMaynadierPrincipality}).

\begin{example} \label{ex-not codim one}
Consider the non-Saito-holonomic example (Example \ref{rmk-Saito-holonomic}) of $f = z(x^{5} + y^{5} + zx^{2}y^{3})$ and $F=(z, (x^{5} + y^{5} + zx^{2}y^{3}))$ from Section 4.2 of of \cite{BahloulOakuLocalBSIdeals}. By loc. cit. $B_{F,0}$ has the zero dimensional components $\{-i, -\frac{7}{5} \}_{i=2}^{5} \cup \{-2, -\frac{8}{5} \}$ and $B_{F,0}^{\textbf{1}}$ is non-principal with three generators.
\end{example}

Now we turn to the second result of the introduction, that is, to comparing $Z(B_{F,\xpoint}^{\textbf{1}})$ and $Z(B_{H,\xpoint}^{\textbf{1}})$ where $F$ and $H$ correspond to different factorizations of $f$. We write elements of $\mathbb{C}^{r}$ as $A=(A_{1}, \dots, A_{r}) \in \mathbb{C}^{r}$ and reserve $\textbf{a}$ for elements of $\mathbb{N}^{r}$. \iffalse Also, write $B_{f,\mathfrak{x}}$ for the ideal in $\mathbb{C}[s]$ generated by the \emph{Bernstein--Sato polynomial} of $f$. \fi We require the following linchpin result from \cite{ZeroLociI}:

\begin{proposition}
\label{prop-ZeroLoci nabla result} \text{\normalfont (Proposition 3.4.3 \cite{ZeroLociI})}
Suppose that $\mathscr{D}_{X,\xpoint}[S]F^{S} / \mathscr{D}_{X,\xpoint}[S]F^{S+\textbf{a}}$ is $(n+1)$-Cohen--Macaulay. Then 
\[
A \in Z(B_{F,\xpoint}^{\textbf{a}}) \text{ if and only if } \frac{\mathscr{D}_{X,\xpoint}[S]F^{S}}{\mathscr{D}_{X,\xpoint}[S]F^{S+\textbf{a}}} \otimes_{\mathbb{C}[S]} \frac{\mathbb{C}[S]}{\mathbb{C}[S] \cdot (s_{1}-A_{1}, \dots, s_{r} - A_{r})} \neq 0.
\]
\end{proposition}

This proposition will let us equate membership in the zero locus of the Bernstein--Sato ideal with behavior of a certain $\mathscr{D}_{X}$-linear map.

\begin{define} (cf. Definition 3.1 \cite{Bath1})
Let $\nabla^{\textbf{a}}: \mathscr{D}_{X, \xpoint}[S] F^{S} \to \mathscr{D}_{X, \xpoint}[S] F^{S}$ be the $\mathscr{D}_{X, \xpoint}$-linear map induced by sending each $s_{k}$ to $s_{k} + a_{k}$ (recall $\textbf{a} = (a_{1}, \dots, a_{r}))$. So $P(S) F^{S} \mapsto P(S+\textbf{a}) F^{s+\textbf{a}}$ where $P(s) \in \mathscr{D}_{X, \xpoint}[S]$. Let $A = (A_{1}, \dots, A_{r}) \in \mathbb{C}^{r}$ and $(S-A) \mathscr{D}_{X, \xpoint}[S] F^{S}$ correspond to the submodule in $\mathscr{D}_{X, \xpoint}[S]F^{S}$ generated by $s_{1} - A_{1}, \dots, s_{r} - A_{r}$. Since $\nabla^{\textbf{a}}$ sends $s_{k} - A_{k}$ to $s_{k} - (A_{k} - a_{k})$, it induces a $\mathscr{D}_{X, \xpoint}$-linear map
\[
\nabla_{A}^{\textbf{a}}: \frac{\mathscr{D}_{X, \xpoint}[S] F^{S}}{(S-A) \mathscr{D}_{X, \xpoint}[S]F^{S}} \to \frac{\mathscr{D}_{X, \xpoint}[S] F^{S}}{(S-(A-\textbf{a})) \mathscr{D}_{X, \xpoint}[S]F^{S}}.
\]
\end{define}

When $\textbf{a} = \textbf{1}$, this map was considered extensively in Sections 3, 4 of \cite{Bath1}. There it was denoted simply by $\nabla_{A}$. The next corollary answers positively a question raised in Section 3 of loc. cit.:

\begin{corollary} \label{cor-three equivalent things nabla}
Suppose that $f$ is strongly Euler-homogeneous, Saito-holonomic, and tame and $F$ corresponds to any factorization $f = f_{1} \cdots f_{r}$. Then the following are equivalent:
\begin{enumerate}[(i)]
\item $A-1 \notin Z(B_{F,\xpoint})$;
\item $\nabla_{A}^{\textbf{1}}$ is injective at $\xpoint$;
\item $\nabla_{A}^{\textbf{1}}$ is surjective at $\xpoint$.
\end{enumerate}
\iffalse
Moreover, $A - 1 \notin Z(B_{F,\xpoint})$ if and only if the natural $\mathscr{D}_{X,\xpoint}$-inclusion $\mathscr{D}_{X,\xpoint}f^{A} \hookrightarrow \mathscr{D}_{X,\xpoint}f^{A - \textbf{1}}$ is an isomorphism.
\fi
\end{corollary}

\begin{proof}
That (i) implies (ii) and (iii) was shown in Proposition 3.2 of \cite{Bath1}; that (ii) implies (iii) is the content of Theorem 3.11 of \cite{Bath1}. We must show (iii) implies (i). But this is immediate from Theorem \ref{thm- F^S / F^S+1 CM}, Proposition \ref{prop-ZeroLoci nabla result}, and the fact that the cokernel of $\nabla_{A}$ at $\xpoint$ is exactly 
\[
\mathscr{D}_{X,\xpoint}[S]F^{S} / \mathscr{D}_{X,\xpoint}[S]F^{S+1} \otimes_{\mathbb{C}[S]} \mathbb{C}[S] / \mathbb{C}[S] \cdot (s_{1}-(a_{1}-1), \dots, s_{r}-(a_{r}-1)).
\]
\end{proof}
\iffalse
The last claim follows from the natural commutative diagram of $\mathscr{D}_{X,\xpoint}$-modules
\[
\begin{tikzcd}
\frac{\mathscr{D}_{X,\xpoint}[S] F^{S}}{(S-A)\mathscr{D}_{X,\xpoint}[S]F^{S}} \arrow[r, twoheadrightarrow] \arrow{d}{\nabla_{A}^{\textbf{1}}}
    & \mathscr{D}_{X,\xpoint}f^{\textbf{a}} \arrow[d, hookrightarrow] \\
\frac{\mathscr{D}_{X,\xpoint}[S] F^{S}}{(S-(A-\textbf{1}))\mathscr{D}_{X,\xpoint}[S]F^{S}} \arrow[r, twoheadrightarrow]
    & \mathscr{D}_{X,\xpoint}f^{\textbf{a} - \textbf{1}}.
\end{tikzcd}
\]
\fi

Now we can begin to compare different factorizations of $f$.

\begin{define} \label{def-coarser}
Suppose $F$ corresponds to the factorization $f = f_{1} \cdots f_{r}$. For a decomposition of $[r]$ as a disjoint union $\sqcup_{1 \leq t \leq m} I_{t}$, write $h_{t} = \prod_{j \in I_{t}} f_{j}$.  If $H$ corresponds to the factorization $f = h_{1} \cdots h_{m}$ of $f$, then we say $H$ is a \emph{coarser} factorization of $F$. Now define $S_{H}$ as the ideal in $\mathbb{C}[S] = \mathbb{C}[s_{1}, \dots, s_{r}]$ generated by $s_{u} - s_{v}$ for $u, v \in I_{t}$ and all choices of $1 \leq t \leq m$. And, finally, let $\Delta_{H}: \mathbb{C}^{m} \to \mathbb{C}^{r}$ be the embedding determined by $\mathbb{C}[S] \to \mathbb{C}[S] / S_{H}:$
\[
\Delta_{H}(a_{1}, \dots, a_{m}) = (b_{1}, \dots, b_{r}) \text{ where } b_{u} = a_{t} \text{ for all } u \in I_{t} \text{ and for all } t.
\]
If $H$ corresponds to the trivial factorization $f = f$, then $\Delta_{H}$ is the diagonal embedding $a \mapsto (a, \dots, a).$
\end{define}

\iffalse
In Lemma 4.20 of \cite{BudurLocalSystems}, Budur shows that if $H$ is a coarser factorization of $F$ then $\Delta_{H}(Z(B_{H, \xpoint}^{\textbf{1}})) \subseteq Z(B_{F,\xpoint}^{\textbf{1}}).$ We will show equality under the hypotheses of Theorem \ref{thm- gen by derivations}. By Corollary \ref{cor-three equivalent things nabla}, it is enough to show $\nabla_{\Delta_{H}(A)}$ is surjective (injective) if and only if $\nabla_{A}$ is surjective (injective) for $A \in \mathbb{C}^{m}$.
\fi

The following results from the fact both $\ann_{\mathscr{D}_{X,\xpoint}[S]} F^{S}$ and $\ann_{\mathscr{D}_{X,\xpoint}[S]}H^{S}$ are generated by derivations, see Theorem 2.29 of \cite{Bath1}:

\begin{proposition} \text{\normalfont (cf. Proposition 2.33, Remark 3.3 \cite{Bath1})} \label{prop-nabla commutative diagram} Suppose $f$ is strongly Euler-homogeneous, Saito-holonomic, and tame and $F$ corresponds to any factorization $f = f_{1} \cdots f_{r}$. If $H$ corresponds to a coarser factorization $f = h_{1} \cdots h_{m}$ and $A \in \mathbb{C}^{m}$ then we have a commutative diagram of $\mathscr{D}_{X,\xpoint}$-modules and maps:
    \begin{equation} \label{diagram- nabla comm diagram}
    \begin{tikzcd}
        \frac{\mathscr{D}_{X,\xpoint}[S]F^{S}}{(S-\Delta_{H}(A))\mathscr{D}_{X,\xpoint}[S]F^{S}} \arrow{r}{\simeq} \arrow{d}{\nabla_{\Delta_{H}(A)}^{\Delta_{H}(\textbf{a})}} 
            & \frac{\mathscr{D}_{X,\xpoint}[S] H^{s}}{(S-A)\mathscr{D}_{X,\xpoint}[S]H^{s}} \arrow{d}{\nabla_{A}^{\textbf{a}}} \\
        \frac{\mathscr{D}_{X,\xpoint}[S]F^{S}}{(S-(\Delta_{H}(A-1)))\mathscr{D}_{X,\xpoint}[S]F^{S}} \arrow{r}{\simeq}
            & \frac{\mathscr{D}_{X,\xpoint}[S] H^{S}}{(S-(A-1))\mathscr{D}_{X,\xpoint}[S]H^{S}}.
    \end{tikzcd}
    \end{equation}
\end{proposition}

\begin{proof}
We first show the horizontal maps are isomorphisms. We will deal with the case $H$ is the trivial factorization $f = f$--the more general case follows similarly. So $A = (a_{1}).$ Note that $\mathscr{D}_{X, \xpoint}[S] \cdot (S - \Delta_{H}(A)) = \mathscr{D}_{X, \xpoint}[S] \cdot (s_{1} - s_{2}, \dots, s_{r-1} - s_{r}, s_{1} - a_{1})$, cf. Definition \ref{def-coarser}. By the product rule, for a logarithmic derivation $\delta$, the image of $\psi_{F,\xpoint}(\delta)$ modulo $\mathscr{D}_{X,\xpoint}[S] \cdot S_{H}$ is precisely $\psi_{H,\xpoint}(\delta).$ Since $\ann_{\mathscr{D}_{X,\xpoint}[S]}F^{S}$ and $\ann_{\mathscr{D}_{X,\xpoint}[S]}H^{S}$ are both generated by derivations by Theorem \ref{thm- gen by derivations}, the horizontal maps are isomorphisms.

That the diagram is commutative follows from the fact $\nabla_{A}^{\textbf{a}}$ (resp. $\nabla_{\Delta_{H}(A)}^{\Delta_{H}(\textbf{a})}$) is induced by sending $H^{S} \mapsto H^{S+\textbf{a}}$ (resp. $F^{S} \mapsto F^{S+\Delta_{H}(\textbf{a})}$).
\end{proof}

In general we only know that $\Delta_{H}(Z(B_{H,\mathfrak{x}}^{\textbf{1}})) \subseteq Z(B_{F,\mathfrak{x}}^{\textbf{1}}) \cap \Delta_{H}(\mathbb{C}^{m})$, cf. Lemma 4.20 of \cite{BudurLocalSystems}. In our setting, the commutative diagram \eqref{diagram- nabla comm diagram} gives us equality:

\begin{theorem} \label{thm-intersecting with diagonal}
Suppose that $f$ is strongly Euler-homogeneous, Saito-holonomic, and tame and $F$ corresponds to any factorization $f = f_{1} \cdots f_{r}$. Further, assume that $H = h_{1} \cdots h_{m}$ corresponds to a coarser factorization of $f$, $\textbf{a} \in \mathbb{N}^{m}$ such that $h_{1}^{a_{1}} \cdots h_{m}^{a_{m}}$ is not a unit, and $A \in \mathbb{C}^{m}$. Then
\begin{equation} \label{eqn-diagonal property}
A \in Z(B_{H,\xpoint}^{\textbf{a}}) \text{ if and only if } \Delta_{H}(A) \in Z(B_{F,\xpoint}^{\Delta_{H}(\textbf{a})}).
\end{equation}
Setting $H = (f)$ and $\textbf{a} = 1 \in \mathbb{N}$ \iffalse and letting $B_{f,\mathfrak{x}} \subseteq \mathbb{C}[s]$ be the ideal generated by the Bernstein--Sato polynomial,\fi we obtain
\[
Z(B_{f,\mathfrak{x}}) = Z(B_{F, \mathfrak{x}}^{\textbf{1}}) \cap \{s_{1} = \cdots = s_{r} \}.
\]
Here points on the diagonal of $\mathbb{C}^{r}$ are naturally identified with points of $\mathbb{C}$.
\end{theorem}

\begin{proof}
Just as in the proof of Corollary \ref{cor-three equivalent things nabla}, the cokernel of $\nabla_{A}^{\textbf{a}}$ is 
\[
\mathscr{D}_{X,\xpoint}[S]H^{S} / \mathscr{D}_{X,\xpoint}[S]H^{S+\textbf{a}} \otimes_{\mathbb{C}[S]} \mathbb{C}[S] / \mathbb{C}[S] \cdot (s_{1} - (A_{1} - a_{1})), \dots, (s_{r} - (A_{r} - a_{r})).
\]
So by Proposition \ref{prop-ZeroLoci nabla result}, the surjectivity of $\nabla_{A}^{\textbf{a}}$ precisely characterizes membership in $Z(B_{H,\xpoint}^{\textbf{a}})$. Now use the commutative diagram \eqref{diagram- nabla comm diagram}.
\end{proof}

\begin{example} \label{ex-diagonal property without good hypotheses}
Returning to our non-Saito-holonomic example $f = z(x^{5} + y^{5} + zx^{2}y^{3})$ and $F=(z, (x^{5} + y^{5} + zx^{2}y^{3}))$ from Section 4.2 of of \cite{BahloulOakuLocalBSIdeals}, Macaulay2 verifies this diagonal property still holds: $Z(B_{F,0}^{\textbf{1}}) \cap \{s_{1} = s_{2} \} = Z(B_{F,0}^{\textbf{1}}). $
\end{example}

\subsection{Hyperplane Arrangements} \text{ }

In this subsection we document some applications of the previous result for hyperplane arrangements. First, using the diagonal property \eqref{eqn-diagonal property} of Theorem \ref{thm-intersecting with diagonal} we will give formulae for the Bernstein--Sato ideals of any free, central, possibly non-reduced hyperplane arrangement (equipped with any factorization) and corresponding estimates when ``freeness'' is replaced with ``tameness.'' Second, we will build up the necessary notation and facts for our formulae for the Bernstein--Sato ideals attached to generic arrangements appearing in the next subsection. Additionally this will let us prove the factorization $F$ into linear forms has a very special property: under the assumption of tameness, the attached Bernstein--Sato ideal equals its radical.

To begin with the applications, suppose $f$ is a free, central hyperplane arrangement in $X = \mathbb{C}^{n}$, cf. Definition \ref{def-log forms}. In Theorem 1 of \cite{MaisonobeFreeHyperplanes}, Maisonobe showed $Z(B_{F,0}^{\textbf{1}})$ is combinatorially determined provided $f$ is reduced and $F$ is a factorization of $f$ into linear forms; in Theorem 1.4 of \cite{Bath2} we generalized this to arbitrary factorizations $F$ of a reduced $f$, certain factorizations $F$ of a non-reduced $f$, and we computed $Z(B_{f,0})$ if $f$ is a power of a reduced, central, free arrangement. We also outlined in Remark 4.28 of loc. cit. how to extend this result to any factorization $F$ of a possibly non-reduced $f$. Now we can do this:

\begin{corollary} \text{\normalfont (cf. Remark 4.28 \cite{Bath2})} \label{cor-free arrangements}
Suppose that $f$ is a free, central, and possibly non-reduced hyperplane arrangement and $F$ corresponds to any factorization $f = f_{1} \cdots f_{r}$. Then $Z(B_{F,0}^{\textbf{1}})$ is combinatorially determined and explicitly given by Theorem 1.4, equation (1.2) of \cite{Bath2}. In particular this gives a combinatorial formula for the roots of the Bernstein--Sato polynomial of $f$.
\end{corollary}

\begin{proof}
By Theorem 1.4 of \cite{Bath2}, there is an explicit combinatorial formula for $B_{H,\mathfrak{x}}^{\textbf{1}}$ for a free, central, possibly non-reduced $f$ provided $H$ is corresponds to a factorization into linear forms. It is immediate (and part of the construction) that if $F$ is another factorization of $f$, then $\Delta_{F}^{-1}(Z(B_{H,0}^{\textbf{1}}))$ agrees with the formula (1.2) of Theorem 1.4 of loc. cit. for $V(B_{F})$ even without the assumption ``unmixed up to units.'' (To use the formula therein, set $f^{\prime} = 1$. And recall that, cf. Convention \ref{convention}, Remark 4.9 of \cite{Bath2}, $V(B_{F})$ in \cite{Bath2} agrees with $Z(B_{F,0}^{\textbf{1}})$ here since the global Bernstein--Sato ideal agrees with the one at the origin by centrality.) \iffalse \cite{Bath2}, $V(I)$ denotes the reduced locus cut out by $I$. Note that, cf. Remark 4.9 of loc. cit., the global multivariate Bernstein--Sato ideal of $F$ agrees with the local one at the origin since $f$ is central. Thus $V(B_{F})$ in loc. cit. agrees with $Z(B_{F,0}^{\textbf{1}})$ here.) \fi The claim follows by \eqref{eqn-diagonal property} of Theorem \ref{thm-intersecting with diagonal}.
\end{proof}

As this paper was being finished, Wu (independently) used in \cite{WuHyperplanes} the same diagonal procedure as in Corollary \ref{cor-free arrangements} to recover a special case of our formula for $Z(B_{f,0})$ from Theorem 1.4 of \cite{Bath2}. Specifically he considered the case of free, central, and reduced arrangements. He obtains the diagonal property \eqref{eqn-diagonal property} by different means, cf. Theorem 1.1 in \cite{WuHyperplanes}. To deal with the non-reduced case in full generality, i.e. to cover the cases not included in our \cite{Bath2}, one needs the duality formula Theorem 3.9 of \cite{Bath2}, and the consequent symmetry formulae Theorem 3.16 and Corollary 3.18 from loc. cit. This is because the duality computations differ slightly in the non-reduced setting.

\iffalse
\begin{remark} \label{rmk- diagonal property in general}
It is easy to adapt the proof of Theorem \ref{thm-intersecting with diagonal} to get the analagous statement for $Z(B_{F,\mathfrak{x}}^{\textbf{a}})$ and $Z(B_{H, \mathfrak{x}}^{\textbf{a}})$ where $\textbf{a} \in \mathbb{N}^{r}$ is arbitrary. The only issue is technical. While it is easy to define 
\[
\nabla_{A}^{\textbf{a}}: \mathscr{D}_{X,\xpoint}[S]H^{S} / (S-A) \mathscr{D}_{X,\xpoint}[S]H^{S} \to \mathscr{D}_{X,\xpoint}[S]H^{S} / (S-(A - \textbf{a})) \mathscr{D}_{X,\xpoint}[S]H^{S}
\]
one must restrict to $A \in \mathbb{C}^{m}$ so that $A$ is compatible both with  We did not include this because the necessary notation is technical: 
\end{remark}
\fi

In \cite{uli}, Walther demonstrated that the roots of the b-function of a hyperplane arrangement many not be combinatorially determined. Theorem \ref{thm-intersecting with diagonal} lets us extend this demonstration to multivariate Bernstein--Sato ideals:

\begin{example} \label{ex-not combinatorial}
Consider Walther's hyperplane arrangements (Example 5.10 \cite{uli}):
\begin{align*}
f = xyz(x + 3z)(x + y + z)(x + 2y + 3z)(2x + y + z)(2x + 3y + z)(2x + 3y + 4z); \\
g = xyz(x + 5z)(x + y + z)(x + 3y + 5z)(2x + y + z)(2x + 3y + z)(2x + 3y + 4z).
\end{align*}
These have the same intersection lattice but different b-functions: $f$ has $-\frac{16}{9}$ as a root; $g$ does not. See also Remark 4.14.(iv) of \cite{SaitoArrangements}. Since $n = 3$, Remark \ref{rmk-log ddifferential forms} shows $f$ and $g$ are tame. Since a factorization $F$ of $f$ (resp. $G$ of $g$) amounts to grouping hyperplanes of $\Variety(f) = \Variety(g)$ into sets, we can regard $F$ and $G$ as equivalent if they correspond to the same choices of hyperplanes. Thus if $Z(B_{F,0}^{\textbf{1}}) = Z(B_{G, 0}^{\textbf{1}})$ for equivalent $F$ and $G$, then $Z(B_{f,0})$ would equal $Z(B_{g,0})$ by \eqref{eqn-diagonal property} of Theorem \ref{thm-intersecting with diagonal}. This is impossible by construction.

\end{example}

It will hereafter be easier to work in the algebraic context. So $X = \mathbb{C}^{n}$, $f$ is a tame, central, and possibly non-reduced hyperplane arrangement equipped with an arbitrary factorization $F$. In this algebraic situation we can define multivariate Bernstein--Sato ideals using the Weyl algebra $\Weyl$ as well as the naturally defined $\Weyl[S]$-module $\Weyl[S]F^{S}$. We will write $B_{F}^{\textbf{a}}$ for the Bernstein--Sato ideal corresponding to the algebraic functional equation using $\Weyl[S]$ and $\textbf{a}$. Because $f$ is central, one can check that the local, analytically defined $B_{F,0}^{\textbf{a}}$ is the same as the global, algebraically defined Bernstein--Sato ideal $B_{F}^{\textbf{a}}$, see, for example, Remark 4.9. of \cite{Bath2}. 

One goal is to provide estimates to $Z(B_{F}^{\textbf{a}})$ that (1) are key to the formula for Bernstein--Sato ideals of generic arrangements and (2) are mild generalizations of the estimates appearing in \cite{Bath2}. (There $\textbf{a} = \textbf{1}$ was exclusively considered.) \iffalse adapt a strategy of Maisonobe's used in Theorem 2 of \cite{MaisonobeFreeHyperplanes} to show that under the above hypotheses, and additionally assuming $f$ is reduced and free, that $B_{F, 0}^{\textbf{\textbf{1}}}$ is principal. We generalized this to the free, non-reduced setting in Corollary 4.20 of \cite{Bath2}. \fi We require notation particular to  arrangements.

\begin{define} \label{def-intersection lattice etc}
For $f$ a central, possibly non-reduced hyperplane arrangement of degree $d$ with a factorization into homogeneous linear terms $f = l_{1} \cdots l_{d}$, we will usually denote $\Variety(f)$ by $\mathscr{A} \subseteq \mathbb{C}^{n}$. We can define the \emph{intersection lattice} $L(\mathscr{A})$ of $f$:
\[
L(\mathscr{A}) = \left\{ \bigcap_{i \in I} \Variety(l_{i}) \mid I \subseteq [d] \right\}.
\]
We say $X$ is an \emph{edge} if $X \in L(\mathscr{A})$; we denote the \emph{rank} of $X$ by $r(X)$. For an edge $X$, let $J(X) \subseteq [d]$ signify the largest subset of the hyperplanes $\{ \Variety(l_{i}) \}$ that contain $X$. (If $f$ is non-reduced, the ``largest'' assumption matters.) Then $X = \cap_{j \in J(X)} \Variety(l_{j})$. To each edge $X$, we associate a subarrangement $f_{X}$ given by 
\[
f_{X} = \prod_{j \in J(X)} l_{j}
\]
We denote the degree of $f_{X}$, that is $\abs{J(X)}$, by $d_{X}$. The edge $X$ is \emph{decomposable} if there is some change of coordinates so that $f_{X}$ can be written as a product of two hyperplanes using disjoint, nonempty, sets of variables. Otherwise $X$ is \emph{indecomposabe}. If $r(X) < n$, then $X$ is automatically decomposable as an arrangement in $\mathbb{C}^{n}$; in this case we decide if $X$ is decomposable after naturally viewing it in $\mathbb{C}^{r(X)}$.

For a different factorization $f = f_{1} \cdots f_{r}$ we can identify (up to a unit) each $f_{k}$ with a collection of the $l_{i}$. Let $S_{k} \subseteq [d]$ be linear forms defining $f_{k}:$
\[
f_{k} = \prod_{i \in S_{k}} l_{i}.
\]
\iffalse
Define 
\[
S_{X,k} = J_{X} \cap S_{k}.
\]
\fi
The factorization $f = f_{1} \cdots f_{r}$ induces a factorization on $f_{X} = f_{X, 1} \cdots f_{X,r}$ with 
\[
f_{X,k} = \prod_{i \in J(X) \cap S_{k}} l_{i}.
\]
Write $\deg f_{k}$ as $d_{k}$, $\deg f_{X,k}$ as $d_{X,k}.$

\end{define}

First we slightly adapt an argument from \cite{MaisonobeFreeHyperplanes} and \cite{Bath2} to find a particular nice element of $B_{F}^{\textbf{a}}$. The proof is essentially the same as in \cite{Bath2}.

\begin{proposition} \text{\normalfont (cf. Proposition 10 \cite{MaisonobeFreeHyperplanes}, Theorem 4.18 \cite{Bath2})} \label{prop- nice element BS ideal}
Suppose that $f$ is a central, possibly non-reduced hyperplane arrangement and $F$ a factorization of $f$ into homogeneous linear terms. Then there exists an $N \in \mathbb{N}$ such that 
\begin{equation} \label{eqn-nice element of BS ideal}
\prod_{\substack{X \in L(\mathscr{A}) \\ X \text{ indecomposable}}} \prod_{\ell = 0}^{N}  \left( \sum_{j \in J(X)} s_{j} + r(X) + \ell \right) \in B_{F}^{\textbf{a}}.
\end{equation}
\end{proposition}

\begin{proof}
First, pick $\textbf{p} = (p, \dots, p) \in \mathbb{N}^{r}$ such that $\textbf{a} \leq \textbf{p}$ and $\textbf{1} \leq \textbf{p}$ termwise. Clearly $B_{F}^{\textbf{p}} \subseteq B_{F}^{\textbf{a}}$ so it suffices to prove the claim for $\textbf{a} = \textbf{p}.$ In Theorem 4.18 of \cite{Bath2} this is shown for $\textbf{a} = \textbf{1}$. The inductive argument on the rank of $f$ used there works here with one slightly non-cosmetic change. By the induction hypothesis, for each $X \in L(\mathscr{A})$ of rank $n-1$, there exists an operator $P_{X}$ of total order at most $k_{X}$ and a polynomial $b_{X} \in \mathbb{C}[S]$ such that $b_{X} \prod_{j \in J(X)} f_{j}^{s_{j}} = P_{X} \prod_{j \in J(X)} f_{j}^{s_{j}+p}.$ In loc. cit. we pick a natural number $m$ greater than the max of all the total order's of the differential operators coming from the inductive hypothesis therein. \iffalse such that $m > \max \{k_{X} \mid X \in L(\mathscr{A}), r(X) = n-1 \}$; \fi Here, if we pick $m > \max \{ k_{X} + p \mid X \in L(\mathscr{A}), r(X) = n-1\}$ then the argument in loc. cit. applies in this case as well. 
\end{proof}

If $f$ is tame, we can use this nice element of $B_{F,0}^{\textbf{a}}$ to obtain a multivariate generalization of Saito's Theorem 1 of \cite{SaitoArrangements}: the roots of the $b$-function of a reduced arrangement $f$ lie in $(-2 + \frac{1}{\deg f}, 0)$. For the reader's sake, we include the result of Theorem 4.11 from \cite{Bath2} about certain hyperplanes necessarily in $Z(B_{F,0}^{\textbf{1}})$ due to tameness.

\begin{corollary} \label{cor- multivariate root bounds}
Suppose that $f$ is a tame, central, and reduced hyperplane arrangement and $F$ corresponds to any factorization $f = f_{1} \cdots f_{r}$. Then $Z(B_{F,0}^{\textbf{1}})$ is a union of hyperplanes,
\begin{equation} \label{eqn-multivariate root subset}
\bigcup_{\substack{ X \in L(\mathscr{A}) \\ X \text{ indecomposable}}} \bigcup_{\ell = 0}^{d_{X} - 1} \{ \sum_{k=1}^{r} d_{X,k}s_{k} + r(X) + \ell = 0\} \subseteq Z(B_{F,0}^{\textbf{1}})
\end{equation}
and
\begin{equation} \label{eqn- multivariate root supset}
Z(B_{F, 0}^{\textbf{1}}) \subseteq \bigcup_{\substack{ X \in L(\mathscr{A}) \\ X \text{ indecomposable}}} \bigcup_{\ell = 0}^{\lceil T_{X} - 1 \rceil} \{ \sum_{k=1}^{r} d_{X,k}s_{k} + r(X) + \ell = 0\} \iffalse  \subseteq Z(B_{F,0}^{\textbf{1}})  \{ \sum_{j \in j(X)} s_{j} + n + \ell = 0 \} \mid \substack{ X \in L(\mathscr{A}) \\ X \text{ indecomposable} \\ 0 \leq l \leq \lfloor T_{X} \rfloor}  \bigcup_{\substack{X \in L(\mathscr{A}) \\ X \text{indecomposable}}} \bigcup_{0 \leq \ell \leq \lfloor T_{X} \rfloor} \{ \sum_{j \in j(X)} s_{j} + n + \ell = 0 \} \fi
\end{equation}
where $d = \deg f$, $d_{X} = \deg f_{X} = \abs{J(X)}$, $d_{X,k} = \deg f_{X,k} = \abs{J(X) \cap S_{k}}$, cf. Definition \ref{def-intersection lattice etc}, and $T_{X} = 2 d_{X} - \frac{d_{X}}{d} - r(X)$. \iffalse  and $\mathcal{P}(-)$ is the power set of $-$. \fi
\end{corollary}

\begin{proof}
Since $f$ is tame, \eqref{eqn-multivariate root subset} follows from Theorem 4.11 of \cite{Bath2}; that $Z(B_{F, 0}^{\textbf{1}})$ is a union of hyperplanes follows from Corollary \ref{cor- tame implies BS principal} and Proposition \ref{prop- nice element BS ideal}. If we replace $T_{X}$ with some very large $N \in \mathbb{N}$, then \eqref{eqn- multivariate root supset} follows from \ref{prop- nice element BS ideal}. The choice of $T_{X}$ follows from Theorem \ref{thm-intersecting with diagonal} and Theorem 1 in \cite{SaitoArrangements}. \iffalse \ref{cor- BS zero locus codim one} and Proposition \ref{prop- nice element BS ideal} implies \eqref{eqn- multivariate root supset} provided we replace each $T_{X}$ with some very large $N \in \mathbb{N}$. The claim then follows from Theorem 1 in \cite{SaitoArrangements} and Theorem \ref{thm-intersecting with diagonal}. \fi
\end{proof}

When $F$ corresponds to a factorization into linear forms, the nice element of $B_{F}^{\textbf{a}}$ from Proposition \ref{prop- nice element BS ideal} is reduced. The same argument strategy from Theorem \ref{thm- purity implies principality} then yields:

\begin{theorem} \label{thm- tame, linear factorizations, BS ideal reduced}
Suppose that $f$ is a tame, central possibly non-reduced hyperplane arrangement, $F$ corresponds to a factorization of $f$ into linear forms, and $f^{\textbf{a}}$ is not a unit. Then $B_{F}^{\textbf{a}} = \rad (B_{F}^{\textbf{a}}).$
\end{theorem}

\begin{proof}
It is equivalent to prove the result on the analytic side, i.e. $B_{F,0}^{\textbf{a}} = \rad (B_{F,0}^{\textbf{a}})$. By Proposition \ref{prop- nice element BS ideal}, we may find $b(S) \in B_{F, 0}^{\textbf{a}}$ such that $b(S)$ cuts out a reduced hyperplane arrangement. By, say, Theorem \ref{thm- purity implies principality}, $\rad (B_{F,0}^{\textbf{a}})$ is principal. Pick a generator $\gamma(s).$ Write $b(S) = \gamma(S) \frac{b(S)}{\gamma(S)}$. Since $\Variety(b(S))$ and $\Variety(\gamma(s))$ are both reduced hyperplane arrangements, every component of $\Variety(\gamma(S)) \cap \Variety(\frac{b(S)}{\gamma(S)})$ has dimension at most $n-2$.

Now consider 
\[
Q := \gamma(S) \frac{\mathscr{D}_{X,0}[S]F^{S}}{\mathscr{D}_{X,0}[S] F^{S+\textbf{a}}} \subseteq \frac{\mathscr{D}_{X,0}[S]F^{S}}{\mathscr{D}_{X,0}[S] F^{S+\textbf{a}}}.
\]
Just as in Theorem \ref{thm- purity implies principality}, $Q$ inherits $(n+1)$-purity from $\mathscr{D}_{X,0}F^{S} / \mathscr{D}_{X,0}F^{S+\textbf{a}}$, is annihilated by $b(S)/\gamma(s)$, and its Bernstein--Sato ideal satisfies $Z(B_{Q}) \subseteq Z(B_{F,0}^{\textbf{a}}) = Z(\gamma(s))$. On one hand, purity implies $Z(B_{Q})$ is purely codimension one; on the other hand, $Z(B_{Q}) \subseteq Z(\gamma(s)) \cap Z(b(s)/\gamma(s))$, which is, by construction, codimension at least two. We deduce $Q = 0$, $\gamma(s) \in B_{F,0}^{\textbf{a}}$, and $B_{F,0}^{\textbf{a}} = \rad (B_{F,0}^{\textbf{a}})$.
\end{proof}

We have now built up enough machinery to deal with generic arrangements.

\subsection{Generic Arrangements} \text{ }

A central arrangement of $d$ hyperplanes $\mathscr{A} \subseteq \mathbb{C}^{n}$ with $d > n$ is \emph{generic} if the intersection of any collection of $n$ hyperplanes is exactly the origin. By \cite{RoseTeraoResGenericLogForms} every generic arrangement is tame. We let $f$ be a reduced defining equation of $\mathscr{A}$ and we continue to use the notation introduced in Definition \ref{def-intersection lattice etc}. \iffalse Recall that by centrality, $B_{F, 0}^{\textbf{a}} = B_{F}^{\textbf{a}}$, cf. Remark 4.9 \cite{Bath2}. \fi

In \cite{WaltherGeneric}, Walther obtained the following formula (see \cite{SaitoArrangements} for the multiplicity of $-1$) for the Bernstein--Sato polynomial of a generic arrangement:
\begin{equation} \label{eqn-walther generic b-function}
    B_{f} = \mathbb{C}[s] \cdot (s+1)^{n-1} \prod_{i=0}^{2d - n -2} (s + \frac{i + n}{d}).
\end{equation}

In \cite{MaisonobeGeneric}, Maisonobe proved that if $F = (f_{1}, \dots, f_{d})$ corresponds to a factorization of $f$ into linear forms, then
\begin{equation} \label{eqn-Maisonobe generic BS member}
    \prod_{k=1}^{d} (s_{k} + 1) \prod_{i=0}^{2d - n -2}(\sum_{k=1}^{d} s_{k} + i + n) \in B_{F}^{\textbf{1}}.
\end{equation}
Moreover, if $d = n+1$ he showed that the polynomial in \eqref{eqn-Maisonobe generic BS member} generates $B_{F}^{\textbf{1}}$. However, computing $B_{F}^{\textbf{1}}$ when $d > n+1$ remained unsolved. Using Walther's formula \eqref{eqn-walther generic b-function}, we independently verify \eqref{eqn-Maisonobe generic BS member} and improve upon it by computing this Bernstein--Sato ideal explicitly:

\begin{theorem} \label{thm- generic BS ideal formulae}
Let $f$ be the reduced defining equation of a central, generic hyperplane arrangement in $\mathbb{C}^{n}$ with $d = \deg f > n$. If $F$ corresponds to a factorization of $f$ into irreducibles, then 
\begin{equation} \label{eqn-my generic BS generator}
B_{F}^{\textbf{1}} = \mathbb{C}[S] \cdot \prod_{k=1}^{d} (s_{k} + 1) \prod_{i=0}^{2d - n -2}(\sum_{k=1}^{d} s_{k} + i + n).
\end{equation}
If $F=(f_{1}, \dots, f_{r})$ corresponds to some other factorization, then $B_{F}^{\textbf{1}}$ is principal and
\begin{equation} \label{eqn-my generic BS zero locus}
Z(B_{F}^{\textbf{1}}) = \left(\bigcup_{k=1}^{r} \{ s_{k} + 1 = 0 \}\right) \bigcup \left( \bigcup_{i = 0}^{2d - n - 2} \{ \sum_{k=1}^{r} d_{k}s_{k} + i + n = 0 \} \right).
\end{equation}
\end{theorem}

\begin{proof}
First assume $F$ corresponds to a factorization of $f$ into irreducibles. By Theorem \ref{thm- tame, linear factorizations, BS ideal reduced}, $B_{0}^{\textbf{1}}$ is principal and equals its radical. Say it is generated by the polynomial $\gamma(S)$. Since at every point other than the origin, $f$ is either smooth or a normal crossing divisor (and hence decomposable), Corollary \ref{cor- multivariate root bounds} implies
\begin{equation} \label{eqn-generic first estimate gamma}
\gamma(S) = \prod_{k=1}^{d} (s_{k} + 1) \prod_{\ell = 0}^{d - 1} (\sum_{k=1}^{d} s_{k} + \ell + n) \prod_{i \in I} (\sum_{k=1}^{d} s_{k} + i + n),
\end{equation}
where $I \subseteq \{ d, d+1, \dots, 2d - n -2  \}.$

By Theorem \ref{thm-intersecting with diagonal} and \eqref{eqn-generic first estimate gamma} we have, after naturally identifying points on the diagonal of $\mathbb{C}^{r}$ with points in $\mathbb{C}$,
\begin{equation} \label{eqn-generic first estimate bfunction}
Z(B_{f, 0}^{\textbf{1}}) = Z(B_{F,0}^{\textbf{1}}) \cap \{ s_{1} = \cdots = s_{r} \} = \{ - \frac{i + n}{d} \mid 0 \leq i \leq d-1 \text{ or } i \in I \}.
\end{equation}
It is well known that the Bernstein--Sato polynomial of a central arrangement at the origin agrees with its global b-function. So \eqref{eqn-generic first estimate bfunction} should agree with the roots computed by Walther's formula \eqref{eqn-walther generic b-function}. Thus $I = \{d, d+1, \dots, 2d - n -2 \}$ and the explicit description of $\gamma(S)$ verifies the first claim \eqref{eqn-my generic BS generator}.

The statement about principality for other factorizations follows from Corollary \ref{cor- tame implies BS principal}. As for \eqref{eqn-my generic BS zero locus}, use \eqref{eqn-my generic BS generator} and Theorem \ref{thm-intersecting with diagonal}.
\end{proof}

\appendix

\section{Logarithmic Differential Forms for Non-Reduced Divisors}

In K. Saito's original \cite{SaitoLogarithmicForms} logarithmic differential forms are considered only for reduced divisors. As the referee points out, there is no systematic taxonomy of the non-reduced analogue. We record some basic observations here in the setting of $X$ smooth analytic (or $\mathbb{C}$-scheme) of dimension $n$, though similar results hold in the algebraic setting. Throughout $D = \sum n_{i}Z_{i}$, $Z_{i}$ irreducible codimension one components, will be an effective divisor, i.e. $n_{i} > 0$, where we have made the sum finite for notational reasons. We allow $D$ to be possibly non-reduced, i.e. some $n_{i}$ may not equal one. When working with the reduced structure we use $D_{\text{red}}.$ Recall $\mathscr{O}_{X}(D) = \{g \in \mathscr{O}_{X}(\star D) \mid \text{div}(g) + D \geq 0\}$, that is, the sheaf of meromorphic functions with at worst poles of order $n_{i}$ along each $Z_{i}$. And $\Omega_{X}^{k}(D) = \Omega_{X}^{k} \otimes_{\mathscr{O}_{X}} \mathscr{O}_{X}(D)$ is the sheaf of meromorphic $k$-forms with poles of order at most $n_{i}$ along each $Z_{i}$. We denote by $\Omega_{X}^{k}(\star D)$ the meromorphic $k$-forms with poles of arbitrary order on (and only on) the $Z_{i}$.

\begin{define}
The logarithmic differential $k$-forms $\Omega_{X}^{k}(\log D)$ is the subsheaf of $\Omega_{X}^{k}(\star D)$ characterized by
\[
\Omega_{X}^{k}(\log D) = \{ g \in \Omega_{X}^{k}(D) \mid d(g) \in \Omega_{X}^{k+1}(D) \},
\]
where $d$ is the exterior derivative inherited from $\Omega_{X}^{\bullet}(\star D)$. The logarithmic differential forms fit into the logarithmic de Rham complex
\[
\Omega_{X}^{\bullet}(\log D) = 0 \to \Omega_{X}^{0}(\log D) \xrightarrow[]{d} \Omega_{X}^{1}(\log D) \xrightarrow[]{d} \cdots \xrightarrow[]{d} \Omega_{X}^{n}(\log D) \xrightarrow[]{d} 0,
\]
which is well-defined since $\Omega_{X}^{\bullet}(\star D)$ is itself a complex. When using an explicit defining equation $f$ for $D$ on a domain $U \subseteq X$ we often replace ``$\log D$'' with ``$\log f$''.
\end{define}

\begin{proposition} \label{prop - appendix, first property log forms}
Suppose $f$ defines $D$ on an open domain $U \subseteq X$. For $\eta \in \Omega_{X}^{k}(\star D)$, the following are equivalent:
\begin{enumerate}
    \item $\eta \in \Omega_{U}^{k}(\log D)$;
    \item $f \eta \in \Omega_{U}^{k}$ and $f d(\eta) \in \Omega_{U}^{k+1}$;
    \item $f \eta \in \Omega_{U}^{k}$ and $df \wedge \eta \in \Omega_{U}^{k}$.
\end{enumerate}
\begin{proof}
This is immediate from the formula $d(f \eta) = d(f) \wedge \eta + f d(\eta)$.
\end{proof}
\end{proposition}

Proposition \ref{prop - appendix, first property log forms} tracks with the first half of (1.1) of \cite{SaitoLogarithmicForms}, but the rest of (1.1) seems to require $D = D_{\text{red}}$. The explicit construction of $g$ in (iii) of loc. cit. amounts to choosing $g = \partial f/ \partial x_{j}$ for some $j$ such that $\{f = 0\} \cap \{ \partial f/ \partial x_{j} = 0\}$ has codimension at least $2$. When $D \neq D_{\text{red}}$ this may not be possible: $\codim \text{Sing}(D) \geq 2$ may fail. For instance, if $f = x^{2} \in \mathscr{O}_{\mathbb{C}^{2}}$ we have
\[
\{f = 0\} \cap \{ \partial f/ \partial x = 0\} = \{f = 0\}
\]
and $\codim \{f = 0\} = 1$.

An unfortunate reality of the inability to generalize the proof of the equivalence of (iii) and (iv) in (1.1) of \cite{SaitoLogarithmicForms} to (2) and (3) of Proposition \ref{prop - appendix, first property log forms} in the non-reduced setting, is that when $D$ is non-reduced, the logarithmic differential forms along $D$ may not be (and almost never are) closed under exterior product.

\begin{example} \label{ex - appendix, not nec exterior algebra} (Not necessarily an exterior algebra)
Let $f = x^{2}y^{2}$ define our divisor $D \subseteq X=\mathbb{C}^{2}$. Then $(1/f)(a dx + b dy)$ is in $\Omega_{X}^{1}(\log f)$ precisely when 
\[
(-(2x^{2} y dy)a + (2x y^{2}) b) dxdy \in f \cdot \Omega_{X}^{2},
\]
that is, when
\[
2xy(by - ax) \in x^{2}y^{2} \cdot \mathscr{O}_{X}.
\]
This happens if and only if $b \in x \cdot \mathscr{O}_{X}$ and $a \in y \cdot \mathscr{O}_{X}$ meaning
\begin{equation} \label{eqn - appendix, explicit log 1 forms in example}
\Omega_{X}^{1}(\log f) = \mathscr{O}_{X} \cdot \frac{dx}{x^{2}y} \oplus \mathscr{O}_{X} \cdot \frac{dy}{xy^{2}}.
\end{equation}
Therefore
\[
\wedge^{2} \Omega_{X}^{1}(\log f) = \frac{1}{x^{3}y^{3}} \Omega_{X}^{2} = \frac{1}{f_{\red} f} \Omega_{X}^{2} \supsetneq \frac{1}{f} \Omega_{X}^{2} = \Omega_{X}^{2}(\log f),
\]
demonstrating that the logarithmic differential forms along a non-reduced divisor may not be an exterior algebra.
\end{example}

Recall the definition of the logarithmic derivations along $D$:
\[
\Der_{X}(- \log D) = \{\delta \in \Der_{X} \mid \delta \bullet \mathscr{O}_{X}(-D) \subseteq \mathscr{O}_{X}(-D)\},
\]
i.e. these are vector fields that are tangent to $D$. It is well known (and follows from the product rule) that $\Der_{X}(-\log D) = \cap_{Z_{i}} \Der_{X}(-\log Z_{i})$ and hence 
\[
\Der_{X}(-\log D) = \Der_{X}(-\log D_{\red}).
\]
In the reduced case, contraction of a logarithmic $k$-form along a logarithmic derivation gives a logarithmic $(k-1)$-form. The same is true in the non-reduced situation:

\begin{proposition} \label{prop - appendix, contraction}
Contracting along $\chi \in \Der_{X}(-\log D) = \Der_{X}(-\log D_{\red})$ induces an $\mathscr{O}_{X}$-map
\[
\iota_{\chi}: \Omega_{X}^{k}(\log D) \to \Omega_{X}^{k-1}(\log D).
\]
\end{proposition}

\begin{proof}
Let $\eta \in \Omega_{U}^{k}(\log D)$ and let $f$ be a defining equation for $D$ along a domain $U$. We use Proposition \ref{prop - appendix, first property log forms}. Certainly $f \iota_{\chi}(\eta) \in \Omega_{U}^{k-1}$ since contraction is $\mathscr{O}_{X}$-linear. We must show $df \wedge \iota_{\chi}(\eta) \in \Omega_{U}^{k}$. Observe:
\[
\iota_{\chi}(df \wedge \eta) = \iota_{\chi}(df) \wedge \eta - df \wedge \iota_{\chi}(\eta).
\]
By Proposition \ref{prop - appendix, first property log forms}, $\iota_{\chi}(df \wedge \eta) \in \Omega_{U}^{k}$; since $\chi \in \Der_{U}(-\log D)$, $\iota_{\chi}(df) \wedge \eta \in \Omega_{U}^{k}.$
\end{proof}

In the non-reduced setting, the logarithmic derivations need not be (and almost never are) the $\mathscr{O}_{X}$-dual of $\Omega_{X}^{1}(\log D)$:

\begin{example} \label{ex - appendix, failure of duality} (Duality differs for non-reduced divisors)
Take $f = x^{2}y^{2}$ and $D = \text{Var}(f) \subseteq X = \mathbb{C}^{2}$ as in Example \ref{ex - appendix, not nec exterior algebra}. We have $\Der_{X}(-\log D) = \Der_{X}(-\log D_{\red})$ is freely generated by $x \partial x$ and $y \partial y$. Using the explicit description of $\Omega_{X}^{1}(\log D)$ from \eqref{eqn - appendix, explicit log 1 forms in example} one can check contracting along $\Der_{X}(-\log D)$ gives a bilinear $\mathscr{O}_{X}$-map
\[
\Omega_{X}^{1}(\log D) \times \Der_{X}(-\log D) \to \mathscr{O}_{X}(D_{\red}).
\]
So $\Der_{X}(-\log D) \neq \Hom_{\mathscr{O}_{X}}(\Omega_{X}^{1}(\log D), \mathscr{O}_{X})$ in this non-reduced case.
\end{example}

Despite Example \ref{ex - appendix, not nec exterior algebra} and Example \ref{ex - appendix, failure of duality}, at least as $\mathscr{O}_{X}$-modules, $\Omega_{X}^{k}(\log D)$ and $\Omega_{X}^{k}(\log D_{\red})$ are closely linked and essentially the same up to isomorphism, thanks to:

\begin{proposition} \label{prop - appendix relating log forms along D to reduced D}
We have the $\mathscr{O}_{X}$-module equality
\[
\Omega_{X}^{k}(\log D) = \Omega_{X}^{k}(\log D_{\red}) \otimes_{\mathscr{O}_{X}} \mathscr{O}_{X}(D - D_{\red}),
\]
under the canonical identification $\mathscr{O}_{X}(D) \otimes_{\mathscr{O}_{X}} \mathscr{O}_{X}(-D) = \mathscr{O}_{X}.$
\end{proposition}

\begin{proof}
It suffices to show that 
\begin{equation} \label{eqn - appendix, sufficient criterion for relating reduced log to non-reduced log after a twist}
\Omega_{X}^{k}(\log D) \otimes_{\mathscr{O}_{X}} \mathscr{O}_{X}(-D) = \Omega_{X}^{k}(\log D_{\red}) \otimes_{\mathscr{O}_{X}} \mathscr{O}_{X}(-D_{\red}).
\end{equation}
Write $D = \sum n_{i}Z_{i}$ as before. Proposition \ref{prop - appendix, first property log forms} implies that $\mu \in \Omega_{X}^{k}(\log D)$ if and only if (1) $\mu \in \Omega_{X}^{k}(D)$ and (2) $dg \wedge \mu \in \Omega_{X}^{k+1}$ for all $g \in \Omega_{X}^{0}(-D)$. But (2) is equivalent, by the product rule, to $dg \wedge \mu \in \Omega_{X}^{k+1}(D - n_{i}Z_{i})$ for all $g \in \Omega_{X}^{0}(-n_{i}Z_{i})$ and for all $Z_{i}$. And this itself is equivalent, again by the product rule, to $dg \wedge \mu \in \Omega_{X}^{k+1}(D - Z_{i})$ for all $g \in \Omega_{X}^{0}(-Z_{i})$ and for all $Z_{i}.$ Therefore
\begin{align} \label{eqn - appendix, criterion for membership in twisted log k forms}
\omega \in \Omega_{X}^{k}(\log D) \otimes_{\mathscr{O}_{X}} \mathscr{O}_{X}(-D) \iff 
    & \omega \in \Omega_{X}^{k} \text{ and } \\
    &dg \wedge \omega \in \Omega_{X}^{k+1}(-Z_{i}) \enspace \forall g \in \Omega_{X}^{0}(-Z_{i}), \enspace \forall Z_{i}. \nonumber
\end{align}

Note that \eqref{eqn - appendix, criterion for membership in twisted log k forms} also applies for $D_{\red}$. In particular, the second condition (after ``and'') is the same. Therefore $\omega \in \Omega_{X}^{k}$ satisfies $\omega \in \Omega_{X}^{k}(\log D) \otimes_{\mathscr{O}_{X}} \mathscr{O}_{X}(-D)$ if and only if $\omega \in \Omega_{X}^{k}(\log D_{\red}) \otimes_{\mathscr{O}_{X}} \mathscr{O}_{X}(-D_{\red})$, which verifies \eqref{eqn - appendix, sufficient criterion for relating reduced log to non-reduced log after a twist}.
\end{proof}

\iffalse

----------
$\mu \in \Omega_{X}^{k}(\star D)$ satisfies $\mu \in \Omega_{X}^{k}(\log D) \otimes_{\mathscr{O}_{X}} \mathscr{O}_{X}(-D)$ if and only if $\mu \in \Omega_{X}^{k}$ and $d(\mu) \in \Omega_{X}^{k+1}$.
\begin{align*}
\Omega_{X}^{k+1} \ni d(\Omega_{X}^{0}(-D)) \wedge \mu 
    &= \left( \sum_{Z_{i}} \left(d(\Omega_{X}^{0}(-n_{i}Z_{i})) \otimes_{\mathscr{O}_{X}} \mathscr{O}_{X}(\sum_{j \neq i} -n_{j}Z_{j}) \right) \right) \wedge \mu \\
    &= \sum_{Z_{i}} \left(d(\Omega_{X}^{0}(-Z_{i})) \otimes_{\mathscr{O}_{X}} \mathscr{O}_{X}(-D + Z_{i}) \wedge -n_{i} \mu \right),
\end{align*}
as justified by the product rule. where the computations are basic manipulations of the product rule. Therefore, this membership is satisfied if and only if 
\begin{equation} \label{eqn - appendix, criterion log k-form multiplying by f}
d(\Omega_{X}^{0}(-Z_{i})) \wedge \mu \in \Omega_{X}^{k+1}(-Z_{i}) \text{ for all } Z_{i}.
\end{equation}
The above argument did not require $D$ to be reduced nor non-reduced, so condition \eqref{eqn - appendix, criterion log k-form multiplying by f} characterizes membership of both $\eta \in \Omega_{X}^{k}(\log D) \otimes_{\mathscr{O}_{X}} \mathscr{O}_{X}(-D)$ and $\eta \in \Omega_{X}^{k}(\log D_{\red}) \otimes_{\mathscr{O}_{X}} \mathscr{O}_{X}(-D_{\red})$.

\fi

\begin{corollary} \label{cor - appendix, listing similar properties for non-reduced}
Let $D$ be a possibly non-reduced divisor and $D_{\red}$ the associated reduced divisor. The following are true:
\begin{enumerate}
    \item $\Omega_{X}^{0}(\log D) = \Omega_{X}^{0} \otimes_{\mathscr{O}_{X}} \mathscr{O}_{X}(D - D_{\red}) = \mathscr{O}_{X}(D - D_{\red});$
    \item $\Omega_{X}^{n-1}(\log D) \simeq \Der_{X}(-\log D)$;
    \item $\Omega_{X}^{k}(\log D)$ is reflexive;
    \item $\Omega_{X}^{k}(\log D)$ is free if and only if $\Omega_{X}^{k}(\log D_{\red})$ is free;
    \item $\Omega_{X}^{k}(\log D)$ is tame if and only if $\Omega_{X}^{k}(\log D_{\red})$ is tame.
\end{enumerate}
\end{corollary}
\begin{proof}
Because $\Omega_{X}^{0}(\log D_{\red}) = \Omega_{X}^{0} = \mathscr{O}_{X}$, Proposition \ref{prop - appendix relating log forms along D to reduced D} implies (1). The justification for (2) is already known in the non-reduced and reduced cases, and is given by a straightforward extension of the explicit map in Remark 2.4 (5) of \cite{uli}.  (4) and (5) follow from the $\mathscr{O}_{X}$-module isomorphism $\Omega_{X}^{k}(\log D) \simeq \Omega_{X}^{k}(\log D_{\red})$ given by Proposition \ref{prop - appendix relating log forms along D to reduced D}. (3) follows similarly, as the reflexivity of $\Omega_{X}^{k}(\log D_{\red})$ was proved for $k=1$ in (1.6) of \cite{SaitoLogarithmicForms} and $k \geq 1$ by Proposition 2.2 of \cite{DenhamSchulzeChern}.
\end{proof}

We finish with the appropriate generalization of the perfect pairing between $\Omega_{X}^{1}(\log D_{\red})$ and $\Der_{X}(-\log D_{\red})$ to the non-reduced setting. We augment a treatment in \cite{MondOnlineLogarithmicNotes}.

\begin{proposition} \label{prop - appendix, perfect pairing}
Contraction along vector fields induces the perfect pairing
\[
\Der_{X}(-\log D) \otimes_{\mathscr{O}_{X}} \mathscr{O}_{X}(D_{\red} - D) \times \Omega_{X}^{1}(\log D) \ni (\chi, \eta) \mapsto \iota_{\chi}(\eta) \in \mathscr{O}_{X}.
\]
This naturally identifies $\Der_{X}(-\log D) \otimes_{\mathscr{O}_{X}} \mathscr{O}_{X}(D_{\red} - D)$ and $\Omega_{X}^{1}(\log D)$ with each other's $\mathscr{O}_{X}$-duals.
\end{proposition}
\begin{proof}
By Proposition \ref{prop - appendix, contraction} and Corollary \ref{cor - appendix, listing similar properties for non-reduced}.(1), contracting gives containments
\begin{equation} \label{eqn - appendix, log derivations twisted contained inside hom of one forms}
\Der_{X}(-\log D) \otimes_{\mathscr{O}_{X}} \mathscr{O}_{X}(D_{\red} - D) \subseteq \Hom_{\mathscr{O}_{X}}(\Omega_{X}^{1}(\log D), \mathscr{O}_{X}).
\end{equation}
and 
\begin{equation} \label{eqn - appendix, log one forms contained inside hom of twisted log der}
\Omega_{X}^{1}(\log D) \subseteq \Hom_{\mathscr{O}_{X}}(\Der_{X}(-\log D) \otimes_{\mathscr{O}_{X}} \mathscr{O}_{X}(D_{\red} - D), \mathscr{O}_{X}).
\end{equation}

We first show the reverse containment for \eqref{eqn - appendix, log one forms contained inside hom of twisted log der}. Note that 
\[
\Der_{X} \otimes_{\mathscr{O}_{X}} \mathscr{O}_{X}(-D) \subseteq \Der_{X}(-\log D) \otimes_{\mathscr{O}_{X}} \mathscr{O}_{X}(D_{\red} - D) \subseteq \Der_{X}
\]
and the two induced cokernels vanish outside of $D$. So the first $\Ext$-module of the two cokernels vanish, meaning dualizing gives inclusions
\[
\Omega_{X}^{1} \subseteq \Hom_{\mathscr{O}_{X}}(\Der_{X}(-\log D) \otimes_{\mathscr{O}_{X}} \mathscr{O}_{X}(D_{\red} - D), \mathscr{O}_{X}) \subseteq \Omega_{X}^{1}(D).
\]
So $\eta \in \Hom_{\mathscr{O}_{X}}(\Der_{X}(-\log D) \otimes_{\mathscr{O}_{X}} \mathscr{O}_{X}(D_{\red} - D), \mathscr{O}_{X})$ can be thought of as a one-form with poles of order at most $n_{i}$ along each $Z_{i}$ of $D$. To show the containment in \eqref{eqn - appendix, log one forms contained inside hom of twisted log der} is an equality, it is enough to do this stalkwise; by Proposition \ref{prop - appendix, first property log forms} it suffices to show $df \wedge \eta \in \Omega_{X,p}^{1}$ for $f$ defining $D$ at $p$. We may write $f = z_{1}^{n_{1}} \cdots z_{m}^{n_{m}}$ where $z_{i}$ defines $Z_{i}$ at $p$. Writing $\eta = (1/f) \sum_{i} a_{i} dx_{i}$, we must show that $df \wedge \eta \in \mathscr{O}_{X,p}$ or, explicitly, that
\[
a_{i} \frac{\partial f}{\partial x_{j}} - \frac{\partial f}{\partial x_{i}} a_{j} \in \mathscr{O}_{X,p} \cdot f
\]
for all $1 \leq i, j \leq n$. As $\mathscr{O}_{X,p}$ is a UFD, this reduces to showing 
\[
n_{k-1}z_{k}^{n_{k}-1} \left( a_{i} \frac{\partial z_{k}}{\partial x_{j}} - \frac{\partial z_{k}}{\partial x_{i}} a_{j} \right) \in \mathscr{O}_{X,p} \cdot z_{k}^{n_{k}}
\]
for all $1 \leq k \leq m$, or, alternatively, that 
\begin{equation} \label{eqn - appendix, final prop supp equation}
a_{i} \frac{\partial f_{\red}}{\partial x_{j}} - \frac{\partial f_{\red}}{\partial x_{i}} a_{j} \in \mathscr{O}_{X,p} \cdot f_{\red}.
\end{equation}
As 
\[
\chi = \frac{\partial f_{\red}}{\partial x_{j}} \partial_{i} - \frac{\partial f_{\red}}{\partial x_{i}} \partial_{j} \in \Der_{X}(-\log f_{\red}),
\]
we see that 
\[
\eta \text{ maps } \frac{f}{f_{\red}} \chi \in \Der_{X}(-\log D) \otimes_{\mathscr{O}_{X}} \mathscr{O}_{X}(D_{\red} - D) \text{ into } \mathscr{O}_{X}.
\]
This verifies \eqref{eqn - appendix, final prop supp equation}, finishing the justification that \eqref{eqn - appendix, log one forms contained inside hom of twisted log der} is actually an equality.

To certify \eqref{eqn - appendix, log derivations twisted contained inside hom of one forms} is an equality it is enough to dualize (into $\mathscr{O}_{X}$) the (now verified) equality in \eqref{eqn - appendix, log one forms contained inside hom of twisted log der}: since $\Der_{X}(-\log D) = \Der_{X}(-\log D_{\red})$ is reflexive (cf. (1.6) of \cite{SaitoLogarithmicForms}), the resultant equality of duals gives equality in \eqref{eqn - appendix, log derivations twisted contained inside hom of one forms}.
\end{proof}

\iffalse

As for verifying \eqref{eqn - appendix, log derivations twisted contained inside hom of one forms} is an equality, we proceed similarly. The containments
\[
\Omega_{X}^{1} \subseteq \Omega_{X}^{1}(\log D) \subseteq \Omega_{X}^{1}(D)
\]
again dualize to
\[
\Der_{X}(-D) \subseteq \Hom_{\mathscr{O}_{X}}(\Omega_{X}^{1}(\log D), \mathscr{O}_{X}) \subseteq \Der_{X}
\]
and it is enough to show equality stalkwise. First note that, with $f = z_{1}^{n_{1}} \cdots f_{m}^{n_{m}}$ as before, we have $d(f_{\red})/f \in \Omega_{X}^{1}(\log f)$ since
\begin{align*}
df \wedge \frac{d(f_{\red})}{f} 
    &= d(z_{1}^{n_{1}} \cdots z_{m}^{n_{m}}) \wedge \frac{d(z_{1} \cdots z_{m})}{z_{1}^{n_{1}} \cdots z_{m}^{n_{m}}} \\
    &= \sum_{i, j} \frac{n_{i} d(z_{i}) z_{1}^{n_{1}} \cdots z_{m}^{n_{m}}}{z_{i}} \wedge \frac{d(z_{j})}{z_{j}(z_{1}^{n_{1}-1} \cdots z_{m}^{n_{m}-1})} \in \Omega_{X,p}^{2}.
\end{align*}
So for $\delta \in \Hom_{\mathscr{O}_{X, p}}(\Omega_{X,p}^{1}(\log D), \mathscr{O}_{X,p})$ we have 
\[
 \iota_{\delta} (\frac{d(f_{\red})}{f}) \in \mathscr{O}_{X,p}.
\]
This means $\delta \in (f/f_{\red}) \cdot \Der_{X,p}(-\log f_{\red}) = (f/f_{\red}) \cdot \Der_{X,p}(-\log f)$ locally everywhere. Altogether this gives 
\[
\Hom_{\mathscr{O}_{X}}(\Omega_{X}^{1}(\log D), \mathscr{O}_{X}) \subseteq \Der_{X}(-\log D) \otimes_{\mathscr{O}_{X}} \mathscr{O}_{X}(D_{\red} - D).
\]

\fi

\iffalse

\begin{remark}
As it may be of independent interest we highlight the fact that $d(f_{\red})/f \in \Omega_{X}^{1}(\log f)$ as demonstrated at the end of the previous proof.
\end{remark}

\fi

\begin{remark}
Note that whether or not $\Omega_{X}^{k}(\log D)$ is an exterior algebra is not actually used in any of the background material on which this paper relies. What are used are the $\mathscr{O}_{X}$-module properties of $\Omega_{X}^{k}(\log D)$, in particular ones involving (2), (3), (4), (5) of Corollary \ref{cor - appendix, listing similar properties for non-reduced}. (For example, while the Liouville complex of \cite{uli} is defined with an exterior product, the actual property invoked is that 
\[
\Omega_{X}^{k}(\log D) \wedge \Omega_{X}^{1} \subseteq \Omega_{X}^{k+1}(\log D)
\]
which is always true by Proposition \ref{prop - appendix, first property log forms}.)
\end{remark}

\bibliographystyle{abbrv}
\bibliography{refs}

\begin{thebibliography}{10}

\bibitem{BahloulOakuLocalBSIdeals}
R.~Bahloul and T.~Oaku.
\newblock Local {B}ernstein-{S}ato ideals: algorithm and examples.
\newblock {\em J. Symbolic Comput.}, 45(1):46--59, 2010.

\bibitem{Bath1}
D.~Bath.
\newblock Bernstein--sato varieties and annihilation of powers.
\newblock {\em Trans. Amer. Math. Soc.}, 373(12):8543–8582, 2020.

\bibitem{Bath2}
D.~{Bath}.
\newblock {Combinatorially determined zeroes of Bernstein--Sato Ideals for tame
  and free arrangements}.
\newblock {\em J. Singul.}, 20:165--204, 2020.

\bibitem{Bjork}
J.-E. Bj\"{o}rk.
\newblock {\em Analytic D-modules and applications}, volume 247 of {\em
  Mathematics and its Applications}.
\newblock Kluwer Academic Publishers Group, Dordrecht, 1993.

\bibitem{BrianconMaynadierPrincipality}
J.~Brian\c{c}on and H.~Maynadier.
\newblock \'{E}quations fonctionnelles g\'{e}n\'{e}ralis\'{e}es:
  transversalit\'{e} et principalit\'{e} de l'id\'{e}al de {B}ernstein-{S}ato.
\newblock {\em J. Math. Kyoto Univ.}, 39(2):215--232, 1999.

\bibitem{BrunsHerzog}
W.~Bruns and J.~Herzog.
\newblock {\em Cohen-{M}acaulay rings}, volume~39 of {\em Cambridge Studies in
  Advanced Mathematics}.
\newblock Cambridge University Press, Cambridge, 1993.

\bibitem{BudurLocalSystems}
N.~Budur.
\newblock Bernstein-{S}ato ideals and local systems.
\newblock {\em Ann. Inst. Fourier (Grenoble)}, 65(2):549--603, 2015.

\bibitem{BudurCohomologySupportLoci}
N.~Budur, Y.~Liu, L.~Saumell, and B.~Wang.
\newblock Cohomology support loci of local systems.
\newblock {\em Michigan Math. J.}, 66(2):295--307, 2017.

\bibitem{ZeroLociI}
N.~{Budur}, R.~{van der Veer}, L.~{Wu}, and P.~{Zhou}.
\newblock {Zero loci of Bernstein-Sato ideals}.
\newblock {\em Invent. math.}, 225:45--72, 2021.

\bibitem{ZeroLociII}
N.~Budur, R.~van~der Veer, L.~Wu, and P.~Zhou.
\newblock Zero loci of {B}ernstein-{S}ato ideals -- ii.
\newblock {\em Selecta Math. (N.S.)}, 27(3), 2021.

\bibitem{DenhamSchulzeChern}
G.~Denham and M.~Schulze.
\newblock Complexes, duality and {C}hern classes of logarithmic forms along
  hyperplane arrangements.
\newblock In {\em Arrangements of hyperplanes---{S}apporo 2009}, volume~62 of
  {\em Adv. Stud. Pure Math.}, pages 27--57. Math. Soc. Japan, Tokyo, 2012.

\bibitem{Gyoja}
A.~Gyoja.
\newblock Bernstein-{S}ato's polynomial for several analytic functions.
\newblock {\em J. Math. Kyoto Univ.}, 33(2):399--411, 1993.

\bibitem{Kashiwara}
M.~Kashiwara.
\newblock Vanishing cycle sheaves and holonomic systems of differential
  equations.
\newblock In {\em Algebraic geometry ({T}okyo/{K}yoto, 1982)}, volume 1016 of
  {\em Lecture Notes in Math.}, pages 134--142. Springer, Berlin, 1983.

\bibitem{Maisonobe}
P.~{Maisonobe}.
\newblock {Filtration Relative, l'Id{\'e}al de Bernstein et ses pentes}.
\newblock {\em Rend. Sem. Mat. Univ. Padova, \text{\normalfont to appear}}.

\bibitem{MaisonobeGeneric}
P.~{Maisonobe}.
\newblock {Id{\'e}al de Bernstein d'un arrangement central g{\'e}n{\'e}rique
  d'hyperplans}.
\newblock {\em arXiv e-prints}, page arXiv:1610.03357, Oct. 2016.

\bibitem{MaisonobeFreeHyperplanes}
P.~{Maisonobe}.
\newblock {L'id{\'e}al de Bernstein d'un arrangement libre d'hyperplans
  lin{\'e}aires}.
\newblock {\em arXiv e-prints}, page arXiv:1610.03356, Oct. 2016.

\bibitem{Malgrange}
B.~Malgrange.
\newblock Le polyn\^{o}me de {B}ernstein d'une singularit\'{e} isol\'{e}e.
\newblock In {\em Fourier integral operators and partial differential equations
  ({C}olloq. {I}nternat., {U}niv. {N}ice, {N}ice, 1974)}, pages 98--119.
  Lecture Notes in Math., Vol. 459. Springer, Berlin, 1975.

\bibitem{MondOnlineLogarithmicNotes}
D.~Mond.
\newblock Notes on logarithmic vector fields, logarithmic differential forms
  and free divisors (online lecture notes), 2012.

\bibitem{RoseTeraoResGenericLogForms}
L.~L. Rose and H.~Terao.
\newblock A free resolution of the module of logarithmic forms of a generic
  arrangement.
\newblock {\em J. Algebra}, 136(2):376--400, 1991.

\bibitem{SabbahI}
C.~Sabbah.
\newblock Proximit\'{e} \'{e}vanescente. {I}. {L}a structure polaire d'un
  {$\mathscr{D}$}-module.
\newblock {\em Compositio Math.}, 62(3):283--328, 1987.

\bibitem{SaitoLogarithmicForms}
K.~Saito.
\newblock Theory of logarithmic differential forms and logarithmic vector
  fields.
\newblock {\em J. Fac. Sci. Univ. Tokyo Sect. IA Math.}, 27(2):265--291, 1980.

\bibitem{SaitoArrangements}
M.~Saito.
\newblock Bernstein-{S}ato polynomials of hyperplane arrangements.
\newblock {\em Selecta Math. (N.S.)}, 22(4):2017--2057, 2016.

\bibitem{WaltherGeneric}
U.~Walther.
\newblock Bernstein-{S}ato polynomial versus cohomology of the {M}ilnor fiber
  for generic hyperplane arrangements.
\newblock {\em Compos. Math.}, 141(1):121--145, 2005.

\bibitem{uli}
U.~Walther.
\newblock The {J}acobian module, the {M}ilnor fiber, and the {$D$}-module
  generated by {$f^s$}.
\newblock {\em Invent. Math.}, 207(3):1239--1287, 2017.

\bibitem{WuHyperplanes}
L.~{Wu}.
\newblock {Bernstein-Sato ideals and hyperplane arrangements}.
\newblock {\em arXiv e-prints}, page arXiv:2005.13502, May 2020.

\end{thebibliography}

\end{document}